\documentclass[11pt]{amsart}
\usepackage{amsfonts}
\usepackage{graphicx}
\usepackage{amsmath}
\usepackage{amsthm}
\usepackage{amssymb}
\usepackage{amscd}
\usepackage{url}
\usepackage{enumitem}
\usepackage[mathscr]{euscript}
\usepackage{mathtools}
\usepackage{tikz-cd}
\usepackage{tikz}

\usepackage[top=1.4in,bottom=1.4in,left=1.4in,right=1.4in]{geometry}

\usepackage[colorlinks,citecolor=blue,linkcolor=red!70!black]{hyperref}
\usepackage[nameinlink]{cleveref}

\usetikzlibrary{arrows,chains,matrix,positioning,scopes}

\newcommand{\R}{\mathbb{R}}
\newcommand{\Z}{\mathbb{Z}}

\newcommand{\N}{\mathbb{N}}

\def\G{{\Gamma}}

\newcommand{\arr}{\rightarrow}

\newcommand{\id}{\text{Id}}

\newcommand{\Aut}{\operatorname{Aut}}

\newcommand{\MCG}{\operatorname{Map}}
\newcommand{\Map}{\operatorname{Map}}
\newcommand{\ExMap}{\operatorname{Map}^\dag}
\newcommand{\PMCG}{\operatorname{PMap}}
\newcommand{\PMap}{\operatorname{PMap}}

\newcommand{\Homeo}{\operatorname{Homeo}}

\newcommand{\Out}{\operatorname{Out}}

\newcommand{\rk}{\operatorname{rk}}

\newcommand{\lilo}{\mathrm{o}}

\newcommand{\PHE}{\operatorname{PHE}}

\let\emptyset\varnothing

\newtheorem{theorem}{Theorem}[section]
\newtheorem{proposition}[theorem]{Proposition}
\newtheorem{corollary}[theorem]{Corollary}
\newtheorem{lemma}[theorem]{Lemma}
\newtheorem{fact}[theorem]{Fact}
\newtheorem{question}[theorem]{Question}
\newtheorem{conjecture}[theorem]{Conjecture}

\theoremstyle{definition}
\newtheorem{definition}[theorem]{Definition}
\newtheorem{example}[theorem]{Example}
\newtheorem{examples}[theorem]{Examples}
\newtheorem{remark}[theorem]{Remark}
\newtheorem{problem}[theorem]{Problem}

\let\oldtocsection=\tocsection

\let\oldtocsubsection=\tocsubsection

\let\oldtocsubsubsection=\tocsubsubsection
\renewcommand{\tocsection}[2]{\hspace{0em}\oldtocsection{#1}{#2}}
\renewcommand{\tocsubsection}[2]{\hspace{1em}\oldtocsubsection{#1}{#2}}
\renewcommand{\tocsubsubsection}[2]{\hspace{2em}\oldtocsubsubsection{#1}{#2}}

\Crefname{theorem}{Theorem}{Theorems}
\Crefname{proposition}{Proposition}{Propositions}
\Crefname{lemma}{Lemma}{Lemmas}
\Crefname{fact}{Fact}{Facts}
\Crefname{corollary}{Corollary}{Corollaries}
\Crefname{conjecture}{Conjecture}{Conjectures}
\Crefname{claim}{Claim}{Claims}
\Crefname{case}{Case}{Cases}
\Crefname{question}{Question}{Questions}

\title[Surfaces PHE to graphs and their DNB maps]{Surfaces proper homotopy equivalent to graphs and their Dehn--Nielsen--Baer maps} 
\author{Ryan Dickmann, Hannah Hoganson, and Sanghoon Kwak}
\date{\today}

\begin{document}  

\pagebreak
\begin{abstract}
    Motivated by the recent work of Algom-Kfir--Bestvina introducing the mapping class group of an infinite graph via proper homotopy equivalences, we give a necessary and sufficient condition for a surface to be properly homotopy equivalent to a graph. We consider second-countable orientable surfaces that are possibly infinite-type and have noncompact boundary. For surfaces proper homotopy equivalent to graphs, we explore the basic properties of the induced map between the mapping class groups of the surface and the graph. We view this induced map as the basis of a Dehn--Nielsen--Baer analog in the setting of infinite-type surfaces.      
\end{abstract} 
\maketitle
\section{Introduction} \label{sec:Introduction}

Let a surface be a connected, orientable, second-countable, 2-manifold with boundary, and let a graph be a connected 1-dimensional CW complex. It is well known that any noncompact surface or any surface with nonempty boundary is homotopy equivalent to a graph.  However, the basic question of which surfaces are \textit{proper} homotopy equivalent (PHE) to a graph has not been addressed. Recall a continuous map is said to be proper when the preimage of any compact set is compact. 

Intuitively speaking, for a surface to be PHE to a graph it needs to have enough boundary to push off onto a graph in a proper manner. We say a separating simple closed curve $\alpha$ in $S$ is \textbf{boundary isolating} if the quotient space $S/\alpha$ is a wedge of surfaces consisting of a surface with boundary and a noncompact surface without boundary. See \Cref{fig:boundaryisolating} for some examples. We will show that having a boundary isolating curve is an obstruction to being PHE to a graph, and in fact we have the following.

\begin{theorem} \label{thm:maincharacterization}
    A surface $S$ is proper homotopy equivalent to a graph if and only if $S$ has nonempty boundary and no boundary isolating curve.
\end{theorem}

The forward direction is proved by comparing the 2nd compactly supported cohomology groups of surfaces and graphs. Since top compactly supported cohomology is trivial for manifolds with boundary, we actually introduce a new proper homotopy invariant we call \textbf{compactly supported cohomology at infinity} (see \Cref{sec:surfacesNotPHEtographs}). For a noncompact manifold, the top compactly supported cohomology at infinity detects ends that are not accumulated by boundary. For the other direction, we build a proper deformation retraction by cutting the surface $S$ into pieces and performing a proper deformation retraction onto a graph in each one. 

Our discussion will largely focus on surfaces of infinite type, meaning surfaces with infinitely generated fundamental group, such as the surfaces in \Cref{fig:boundaryisolating}. We remark that every surface with nonempty boundary and no boundary isolating curve has noncompact boundary. By allowing an infinite-type surface $S$ to have noncompact boundary, we are considering a larger class of surfaces than much of the recent literature on infinite-type surfaces. For example, not only can there be infinitely many compact boundary components accumulating to an end, but there can be noncompact boundary components homeomorphic to $\R$ with complicated behavior. See \Cref{ssec:prelim_noncompactboundary} for more examples and the classification for surfaces with noncompact boundary components.

\begin{figure}[ht!]
    \centering
    \includegraphics[width=0.8\textwidth]{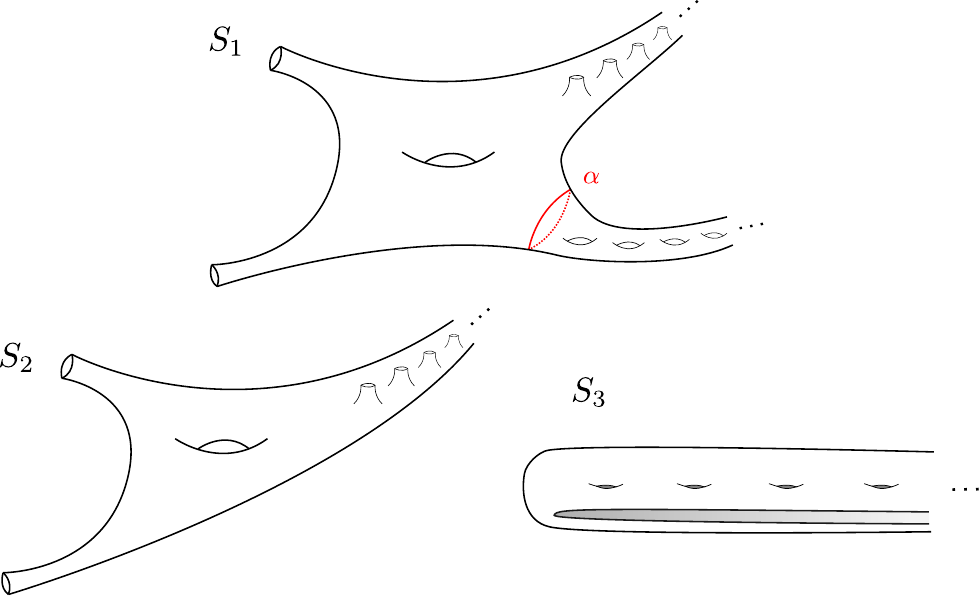}
    \caption{The surface $S_1$ on the top has a boundary isolating curve $\alpha$.
    The two surfaces $S_2$ and $S_3$ on the bottom do not have boundary isolating curves:
    Every noncompact component of the complement of a separating curve contains boundary.}
    \label{fig:boundaryisolating}
\end{figure}

\textbf{Mapping class groups.} The mapping class group of a surface $S$ with nonempty boundary is
\[
    \Map(S) := \Homeo(S, \partial S)/\sim,
\]
where $\Homeo(S, \partial S)$ is the group of homeomorphisms of $S$ pointwise fixing the boundary of $S$, and $f \sim g$ for $f,g \in \Homeo(S,\partial S)$ when $f$ is isotopic to $g$ relative to $\partial S$. Note a homeomorphism fixing the boundary must also preserve orientation. If $\partial S =\emptyset$, then we write $\Homeo^{+}(S)$, the group of orientation-preserving homeomorphisms, in the definition of the mapping class group. When $S$ has noncompact boundary components, its mapping class group was first systematically studied by the first author \cite{dickmann2023mapping}. Recently, he classified the \emph{pure} mapping class groups with automatic continuity including mapping class groups of surfaces with noncompact boundary \cite{dickmann2023automatic}.

Following Algom-Kfir--Bestvina \cite{AB2025} we define the mapping class group of a graph $\G$ as
\[
    \Map(\G) := \PHE(\G)/\sim,
\]
where $\PHE(\G)$ is the group of proper homotopy equivalences of $\G$. Recall a \textit{proper homotopy equivalence} or PHE is a homotopy equivalence $f$ that is proper with a proper inverse $g$ such that $fg$ and $gf$ are properly homotopic to the identity. For $f,g \in \PHE(\G)$, we let $f \sim g$ when $f$ is properly homotopic to $g$. 

The mapping class group for a graph or surface is a topological group with a natural topology. This topology is a quotient of the respective compact-open topology on either the group of PHEs of the graph or the homeomorphism group of the surface. See \cite[Section 4]{AB2025} and \cite{vlamis2019notes} for more details. 
When a proper homotopy equivalence exists between a surface and a graph, it naturally determines a homomorphism between their mapping class groups. We show the following basic properties of the induced map.

\begin{theorem}\label{thm:MainTHM_DNB} 
    Let $S$ be an infinite-type surface properly homotopy equivalent to a graph $\G$. Then the induced homomorphism \[\Theta: \Map(S)\arr \Map(\Gamma)\] is continuous and has a closed image. Moreover, the kernel is topologically generated by Dehn twists about the compact boundary components. 
\end{theorem} 

We remark that as a corollary to \Cref{thm:MainTHM_DNB} we have the mapping class group of a Sliced Loch Ness Monster surface $S_{SLN}$ (for example the bottom-right surface in \Cref{fig:boundaryisolating}) is not residually finite. This follows from the observation that $S_{SLN}$ is PHE to the Loch Ness Monster graph $\G_{LN}$, that $\Map(\G_{LN})$ is not residually finite by \cite[Theorem D]{domat2025generating}, and that $\Theta$ is injective when $S$ has no compact boundary components. This was also shown by the first author \cite{dickmann2023automatic} who proved that $\Map(S_{SLN})$ is automatically continuous. As a corollary, he shows it has no nontrivial homomorphisms to a countable group and hence is not residually finite.

We also remark \Cref{thm:maincharacterization} says any infinite-type surface with compact boundary is not properly homotopic to a graph, so \Cref{thm:MainTHM_DNB} does not apply. This is the class of surfaces much of the recent literature on big mapping class groups has focused on, so we propose the following open-ended question.

\begin{question}\label{q:moreDNB}
    When $S$ is an infinite-type surface with (possibly empty) compact boundary, does there exist any geometrically meaningful homomorphism from $\Map(S)$ to $\Map(\G)$ where $\G$ is a locally finite, infinite graph?
\end{question}

On the other hand, one can ask for a characterization of the image of the map $\Theta$ in \Cref{thm:MainTHM_DNB}.

\begin{question}\label{q:DNBimage}
    Is there an algebraic/geometric description of the image of $\Theta$ with respect to $\G$?
\end{question}

One natural guess for the image is the following. 
Fix $\phi: S \to \G$ a proper homotopy equivalence.
Let $\mathcal{B}$ denote the set of properly immersed oriented curves $S^1 \to \G$ and immersed oriented lines $\R \to \G$ in $\G$ that are images of the components of $\partial S$ via $\phi$. Denote by $\Map(\G, \mathcal{B}) < \Map(\G)$ the subgroup of mapping classes that fix each line or curve in $\mathcal{B}$ and its orientation up to homotopy. 

\begin{conjecture}\label{conj:mainimage}
 The image of $\Theta$ is $\Map(\G, \mathcal{B})$.
\end{conjecture}

In \Cref{sec:image}, we verify that this conjecture holds if one restricts themselves to a subgroup of the domain and image where elements are the limit of compactly supported elements. We also discuss another possible restriction on the image when one considers elements that are not the limit of compactly supported maps. In particular, elements in the image of $\Theta$ have even flux in the sense of Domat--Hoganson--Kwak \cite{DHK2023,domat2025generating}.

\textbf{Extended mapping class groups.}
We note that the image of $\Theta$ lies in $\PMap(\G)$, the pure mapping class group of $\G$, the subgroup of elements fixing ends pointwise. This is because $\Map(S) = \PMap(S)$ when $S$ has no boundary isolating curve as every end of $S$ is accumulated by boundary, and elements in $\Map(S)$ fix boundary components pointwise. Since $\Theta$ is an equivariant homomorphism under the action on the space of ends of $S$ and $\G$, the $\Theta$-image of $\PMap(S)=\Map(S)$ lies in $\PMap(\G)$. 

\begin{figure}[ht!]
    \centering
    \includegraphics[width=0.75\textwidth]{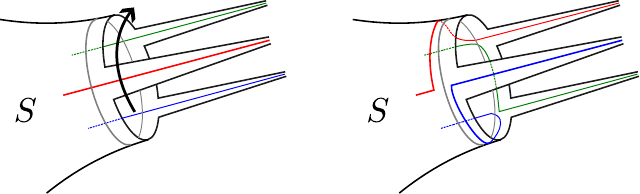}
    \caption{$\frac{1}{3}$-Dehn twist}
    \label{fig:fractional_DT_simpler}
\end{figure}

To achieve a map that hits more elements of the mapping class group $\Map(\G)$, we consider the \textbf{(boundary) extended mapping class group} $\ExMap(S)$ of a surface $S$ defined as:
\[
    \ExMap(S) := \Homeo(S)/\sim,
\]
where $\Homeo(S)$ is the group of homeomorphisms on $S$ (not necessarily fixing boundary pointwise or orientation preserving), and for $f,g \in \Homeo(S)$, we let $f \sim g$ when $f$ is isotopic to $g$ (not necessarily rel boundary). See \Cref{ex: dihedral} for some cases where replacing isotopy with homotopy or proper homotopy can change the isomorphism type of the extending mapping class group. 
When $\partial S = \emptyset$, then $\ExMap(S)$ coincides with the (usual) extended mapping class group $\Map^\pm(S)$, introduced in, e.g., \cite[Section 8]{primer}.

Dehn twists about compact boundary components are isotopic to the identity in $\ExMap(S)$ as we are allowed to `undo the twist' by an isotopy along the boundary. For surfaces with noncompact boundaries, we have some new interesting maps in the extended mapping class group. Consider a surface formed by deleting a set of points from a compact boundary component. We will refer to the union of the resulting boundary components as a \textit{degenerate chain}. For more detail, see \Cref{exa:diskminusboundarypoints} and the discussion thereafter. Now we have Dehn twists about degenerate chains, and as shown in \Cref{fig:fractional_DT_simpler}, we can even have \emph{fractional Dehn twists} about them, which cyclically permute the ends in the degenerate chain. Disarlo \cite{disarlo2015combinatorial} also studied surfaces with points removed from the boundary (called \emph{ciliated surfaces} $(S,\mathbb{P})$ where $\mathbb{P}$ is a finite set of marked points on the boundary of $S$) and made the observation that the extended mapping class group contains fractional Dehn twists. 

In this context, a proper homotopy equivalence $S \to \G$ still induces a well-defined homomorphism between mapping class groups, and we obtain a similar theorem to \Cref{thm:MainTHM_DNB} for the extended mapping class group as follows.
\begin{theorem}\label{thm:MainTHM_DNB_EX}
    Let $S$ be an infinite-type surface properly homotopy equivalent to a graph. Then the induced homomorphism $\widetilde{\Theta}$ from the proper homotopy equivalence
    \[\widetilde{\Theta}: \ExMap(S)\arr \Map(\Gamma)\] is continuous and has a closed image. Moreover, it is injective.
\end{theorem}

One can also ask \Cref{q:moreDNB,q:DNBimage} about the homomorphism $\widetilde{\Theta}$ in \Cref{thm:MainTHM_DNB_EX}. For the set $\mathcal{B}$, as defined in \Cref{conj:mainimage}, of properly immersed oriented curves and lines in $\G$ coming from $\partial S$, denote by $\Map^\star(\G, \mathcal{B}) < \Map(\G)$ the subgroup of mapping classes that fix the set of lines and curves up to homotopy (possibly permuting classes and reversing orientations). Then we have the following conjecture, analogous to \Cref{conj:mainimage}.

\begin{conjecture}\label{conj:mainimage_EX}
 The image of $\widetilde{\Theta}$ is $\Map^\star(\G, \mathcal{B})$.
\end{conjecture}

\textbf{Dehn--Nielsen--Baer Theory.} We now discuss the motivation for the investigation of the homomorphisms between mapping class groups of surfaces and mapping class groups of graphs. The initial object of interest is the natural homomorphism $\sigma$ from a mapping class group of a surface $S$ to $\Out(\pi_1(S))$ determined by the action of a homeomorphism of $S$ on $\pi_1(S)$, up to changing base point. For closed surfaces possibly with finitely many points removed, the well-known Dehn--Nielsen--Baer theorem describes the relationship as follows. Let $S_{g,p}$ denote the unique surface with $g$ genus and $p$ punctures.

\begin{theorem}[{Dehn--Nielsen--Baer, \cite[Section 8]{primer}}]
\label{thm:classic_DNB}
Let $S$ be a compact closed surface with genus $g\geq 1$. The homomorphism \[\sigma: \ExMap(S)\arr \Out(\pi_1(S))\] is an isomorphism. If $S=S_{g,p}$ with $g \geq 1$ or $p \geq 4$, then $\sigma$ is injective, but not surjective when $p \neq 0$. Moreover, the image of $\sigma$ is the subgroup of outer automorphisms that preserve the set of conjugacy classes in $\pi_1(S)$ of the $2p$ oriented curves bounding a single puncture.
\end{theorem} 

The conditions on $g$ and $p$ are to remove sporadic cases. For surfaces with boundary, a similar statement holds where we additionally need the conjugacy classes coming from boundary curves to be preserved. Moreover, the kernel of $\sigma$ is no longer trivial, and it consists of Dehn twists about compact boundary components. Note when a compact surface has nonempty boundary, it is homotopy
equivalent to a finite graph $\G$. Thus, its fundamental group is the free group $F_n$ of finite rank $n=2g+p-1$, and \Cref{thm:classic_DNB} gives a homomorphism $\Map(S)$ to $\Out(F_n)$.

We remark that the existence of the map $\sigma$ in \Cref{thm:classic_DNB} is immediate. However, characterizing the image of $\sigma$ (in particular proving surjectivity when $S$ is closed) requires nontrivial work. In comparison with our contribution in \Cref{thm:MainTHM_DNB,thm:MainTHM_DNB_EX}, establishing the maps $\Theta$ and $\widetilde{\Theta}$ themselves already requires a substantial amount of work. These maps are built for the class of surfaces that are PHE to graphs. In some sense, this additional difficulty is inevitable, since we require more information than how each homeomorphism $h$ acts algebraically on the $\pi_1$, as in the classical Dehn--Nielsen--Baer theorem; we do not discard all the topology of $S$.
Moreover, as described in \Cref{q:DNBimage}, a complete characterization of the image of the DNB maps remains open.

For a finite graph $\G$, it is well known that $\Map(\G)$ is naturally isomorphic to $\Out(F_n)$ via the homomorphism determined by the action of a proper homotopy equivalence on $\pi_1(\G)$. Therefore, for compact surfaces with boundary we can view the Dehn--Nielsen--Baer homomorphism $\sigma$ as the homomorphism from $\ExMap(S)$ to $\Map(\G)$ where it agrees with the map $\widetilde{\Theta}$ induced from the PHE from $S$ to $\G$. 

Since $\Out(F_n)$ is extensively studied in analogy with the mapping class groups of closed surfaces, we initiate the investigation of the relationship between mapping class groups of general noncompact surfaces with the mapping class groups of infinite graphs due to Algom-Kfir--Bestvina \cite{AB2025}. Since homotopy classes of oriented curves in a surface $S$ are in one-to-one correspondence with conjugacy classes in $\pi_1(S)$, we can view \Cref{thm:MainTHM_DNB_EX} and \Cref{conj:mainimage_EX} as extensions to \Cref{thm:classic_DNB}. 

    A more classical Dehn--Nielsen--Baer (DNB) map $\Map(S) \to \Out(F_\infty)$ can also be studied, where $S$ is an infinite-type surface and $\pi_1(S)$ is identified with $F_\infty$, the free group with a countable generating set. This map agrees with the composition of $\Theta$ with an analogous natural map $\Map(\Gamma) \to \Out(F_\infty)$ studied by Algom-Kfir--Bestvina \cite{AB2025}. In particular, they show the map is continuous with natural topologies on both groups, but the image is not a closed subgroup for infinite rank graphs. The map to $\Map(\Gamma)$ is potentially more fruitful as, notably, the image is a closed subgroup. An interpretation of this is that the topology on $\Out(F_\infty)$ is in a sense weaker than the topology on $\Map(\Gamma)$. Additionally, $\Out(F_\infty)$ contains no information about the ends of the original space, and there are no immediate geometric objects to study such as lines up to proper homotopy in graphs.

    Having the DNB map as an isomorphism has another interpretation that \emph{every homotopy equivalence $S\to S$ is homotopic to a homeomorphism.} 
    Das in \cite{das2024strong} showed this is not the case when $S$ is a noncompact surface without boundary, by providing an uncountable set of pairwise nonhomeomorphic surfaces that are all homotopy equivalent. Furthermore, Das--Gadgil--Nair in \cite{das2025goldman} characterized when a homotopy equivalence between noncompact surfaces without boundary is homotopic to a homeomorphism using the Goldman bracket.

As a different kind of homomorphism to $\Map(\G)$, we also remark that recently Udall \cite{udall2024the-sphere} constructed a map $\Psi: \Map(M_\G) \to \Map(\G)$ using a retraction $M_\G \to \G$, where $M_\G$ is the double handlebody constructed from the double of the thickening of $\G$. Extending \cite{laudenbach1974topologie,hatcher2004homology,brendle2023the-mapping} to infinite graphs $\G$, he showed that $\Psi$ is surjective and the kernel of the map $\Psi$ is generated by the Dehn twists about core spheres (corresponding to the core subgraph of $\G$, the smallest connected subgraph of $\G$ containing all immersed images of $S^1$). He further showed that the corresponding short exact sequence splits.

\textbf{Outline.} In \Cref{sec:preliminaries}, we cover background for infinite graphs and infinite-type surfaces with boundary, and their classifications with many examples. Though our arguments will not often use the classification of general surfaces due to Brown--Messer \cite{classification}, we give some discussion on the classification since the general case is still largely unknown.

In \Cref{sec:surfacesNotPHEtographs}, we introduce the boundary isolating curve and show that it is an obstruction for a surface to be properly homotopy equivalent to a graph, showing the forward direction of \Cref{thm:maincharacterization}. We introduce compactly supported cohomology at infinity to accomplish this.

In \Cref{sec:surfacesPHEtograph}, we prove the backward direction of \Cref{thm:maincharacterization}, showing that when a surface has boundary but no boundary isolating curve then the surface is properly homotopy equivalent (in fact, there exists a \emph{proper deformation retraction}) to a graph. With this proper homotopy equivalence in hand, in \Cref{sec:DNB} we prove the Dehn--Nielsen--Baer type theorems \Cref{thm:MainTHM_DNB,thm:MainTHM_DNB_EX}.

In \Cref{sec:image}, we give some partial results towards \Cref{q:DNBimage}. In particular, we show that the image of the induced map $\Theta$ is in the even flux subgroup (\Cref{prop:doubleflux}) and that \Cref{conj:mainimage} holds when we restrict to the closure of the compactly supported mapping class group (\Cref{thm:image_compactsupp}).

\section*{Acknowledgements}
We thank Mladen Bestvina for suggesting this project, and sharing numerous helpful insights. We acknowledge George Domat for useful discussions to kick off this project, in particular regarding \Cref{sec:surfacesNotPHEtographs} and \Cref{sec:surfacesPHEtograph}. We are grateful to the referee for the helpful comments and insightful suggestions that improved the quality of our paper.

The first author was supported by NSF DMS--1745583. The second author was
partially supported by NSF DMS--2303365. The third author was partially
supported by NSF DMS--2304774, the KIAS Individual Grant (HP098501) via the June
E Huh Center for Mathematical Challenges at Korea Institute for Advanced Study,
by the New Faculty Startup Fund (700-20250069) at Seoul National University, and
by the Young Scientist Grant (RS-2026-25481514) of the
National Research Foundation of Korea (NRF), funded by the Korea government
(MSIT).

\tableofcontents

\section{Preliminaries} \label{sec:preliminaries}

\subsection{Infinite graphs}

A graph is a connected 1-dimensional CW complex. An \textbf{infinite} graph is a locally finite graph that is infinite as a CW complex. Namely, it has infinitely many 0-cells (vertices) and 1-cells (edges) with every vertex having finite valence. The \textbf{rank} $\rk(\G)$ of a graph $\G$ is the rank of its fundamental group, a free group. The rank itself is a complete homotopy invariant for a finite graph. Indeed, two finite graphs are homotopy equivalent if and only if they have the same rank. However, for an infinite graph, we need more information to appropriately determine their \emph{equivalence}.

A graph is an example of a \textbf{generalized continuum}, a topological space that is locally
compact, locally connected, $\sigma$-compact, and connected Hausdorff. A
generalized continuum $X$ admits a compact exhaustion
$\{K_n\}_{n=1}^\infty$ of $X$ such that $K_n \subset \textrm{Int}(K_{n+1})$. The
\textbf{(Freudenthal) end space} $E(X)$ (due to \cite{freudenthal1931enden}) of a generalized continuum $X$ is the inverse limit
\[
    E(X) := \varprojlim_{K \subset X} \pi_0(X \setminus K),
\]
where $K$ runs over the compact subsets of $X$, and $\pi_0(X \setminus K)$ is
the set of path components of $X \setminus K$. Each element of $E(X)$ is called
a \textbf{end} of $X$. See \cite[Section I.9]{baues2001infinite} for more
details. See also \cite[Section 2.1]{udall2024the-sphere}.

 Now consider a locally finite, infinite graph $\G$ and consider its end space $E(\G)$.
 If the sequence of complementary components corresponding to $e \in E(\G)$ consists of infinite rank components, then we say $e$ is \textbf{accumulated by loops}. We denote by $E_\ell(\G)$ the set of ends accumulated by loops. With the standard inverse limit topology on $E(\G)$, the set $E_\ell(\G)$ is a closed subset of $E(\G)$.

Recall a continuous map is called \textbf{proper} if every compact set has compact preimage.
Ayala--Dominguez--M{\'a}rquez--Quintero \cite{ayala1990proper} classified infinite graphs up to proper homotopy equivalences using their ranks and end spaces. A homeomorphism between pairs of spaces $(A,B) \to (C,D)$ with $A \supset B$ and $C \supset D$ is a homeomorphism $A \to C$ that restricts to a homeomorphism $B \to D$.

\begin{theorem}[{\cite[Theorem 2.7]{ayala1990proper}}] \label{thm:PHEclassification}
    Let $\G$ and $\G'$ be infinite graphs. If $\rk \G = \rk \G'$ and there is a homeomorphism $(E(\G),E_\ell(\G)) \to (E(\G'), E_\ell(\G'))$, then the homeomorphism extends to a proper homotopy equivalence $\G \to \G'$. If $\G$ and $\G'$ are trees, then this extension is unique up to proper homotopy.
\end{theorem}

This classification motivates using proper homotopy equivalence (PHE) as the notion of symmetry of an infinite graph. Indeed, Algom-Kfir--Bestvina \cite{AB2025} introduced the \textbf{mapping class group} of a locally finite infinite graph $\G$ as:
\[
    \Map(\G) := \{\phi: \G \to \G \textrm{ PHE} \}/\{\textrm{proper homotopy}\}.
\]
More precisely, they defined the mapping class group of $\G$ as the group of proper homotopy equivalences that necessarily have a proper homotopy inverse up to proper homotopy. See \cite[Example 4.1]{AB2025} for an example of a homotopy equivalence which itself is proper, but whose homotopy inverse is never proper.
Hence, such an element should not be included in the mapping class group.

A proper homotopy equivalence induces a homeomorphism between their end spaces.
\begin{proposition}[{\cite[Section 2]{porter1995proper}, see also \cite[Section 13.4]{geoghegan2008topological}}]
\label{prop:InducedMapsOnEnds}
    A proper map $f:X \to Y$ between connected, locally compact, Hausdorff spaces induces a continuous map $Ef:E(X) \to E(Y)$ between the end spaces. Moreover, two properly homotopic proper maps $f,g:X \to Y$ induce the same continuous map $Ef = Eg$ on $E(X)$. In particular, a proper homotopy equivalence $f:X \to Y$ induces a homeomorphism $Ef: E(X) \to E(Y)$. Also, $E$ is functorial; $E(f \circ g) = E(f) \circ E(g)$, for proper maps $g: X \to Y$ and $f: Y \to Z$.
\end{proposition}

In fact, \Cref{prop:InducedMapsOnEnds} extends to the end space pair $(E,E_\ell)$; a proper homotopy equivalence with proper inverse on an infinite graph induces a homeomorphism on the pair $(E, E_\ell)$. Also, properly homotopic proper homotopy equivalences induce the same homeomorphism on the pair, so this induces the action of $\Map(\G)$ on $(E(\G), E_\ell(\G))$ by homeomorphism. We call the kernel of this action the \textbf{pure mapping class group} of $\G$, and denote it by $\PMap(\G)$. In short, it fits into the short exact sequence
\[
1 \longrightarrow \PMap(\G) \longrightarrow \Map(\G) \longrightarrow \Homeo(E(\G),E_\ell(\G)) \longrightarrow 1
\]

We record here the pure mapping class group of a \emph{finite type} graph for
the future reference.

\begin{lemma}[{\cite[Corollary 3.9]{AB2025} and \cite[Proposition 2.3]{domat2025generating}}]
  \label{lem:PMapfinitegraph}
  Let $\G$ be a graph with finite rank $n$ and finite end space with $e$ ends.
  Then
  \[
    \PMap(\G) \cong
    \begin{cases}
      \Out(F_n) & \text{if $e=0$,}\\
      F_n^{e-1} \rtimes \Aut(F_n) & \text{if $e>0$.}
    \end{cases}
  \]
\end{lemma}

\subsection{Surfaces with compact boundary}\label{ssec:prelim_compactboundary}
As mentioned in the introduction, a surface is a connected, orientable,
second-countable 2-manifold. We first focus on those with compact boundary. The
end space $E(S)$ of a surface $S$ with its natural topology is defined
analogously to the end space of a graph by considering complementary components
of compact subsurfaces. An end is said to be \textbf{accumulated by genus} when
any complementary component corresponding to the end has infinite genus. The set
of ends accumulated by genus is denoted by $E_\infty(S)$. With the standard
inverse limit topology on $E(S)$, the set $E_\infty(S)$ is a closed subset of
$E(S)$. Note $E_\infty(S) = \emptyset$ exactly when the surface has finite
genus. Due to the work of Ker{\'e}kj{\'a}rt{\'o} \cite{Kerekjarto1923} and
Richards \cite{Richards1963}, the homeomorphism type of a surface with compact
boundary is essentially determined by these two sets as well as its genus and the
number of compact boundary components.

\begin{theorem} [\cite{Kerekjarto1923,Richards1963}]
    Two surfaces with compact boundary are homeomorphic if and only if the pairs $(E(S), E_\infty(S))$ are homeomorphic and they have the same genus and the same number of boundary components.
\end{theorem}

\subsection{Surfaces with noncompact boundary}
\label{ssec:prelim_noncompactboundary}

First, we introduce some examples that highlight the new phenomena in the noncompact boundary case. 

\begin{example} [Disks with boundary points removed] \label{exa:diskminusboundarypoints}
 Take a disk and remove an embedded closed subset $I$ of the Cantor set from the boundary. In the case that $I$ is homeomorphic to a Cantor set, this gives a surface with countably many noncompact boundary components yet uncountably many ends.
 
 Each of these surfaces can be constructed by taking a small closed neighborhood of an infinite, locally finite, tree properly embedded in $\R^2$.
\begin{figure}[ht!]
    \centering
    \includegraphics[width=0.6\textwidth]{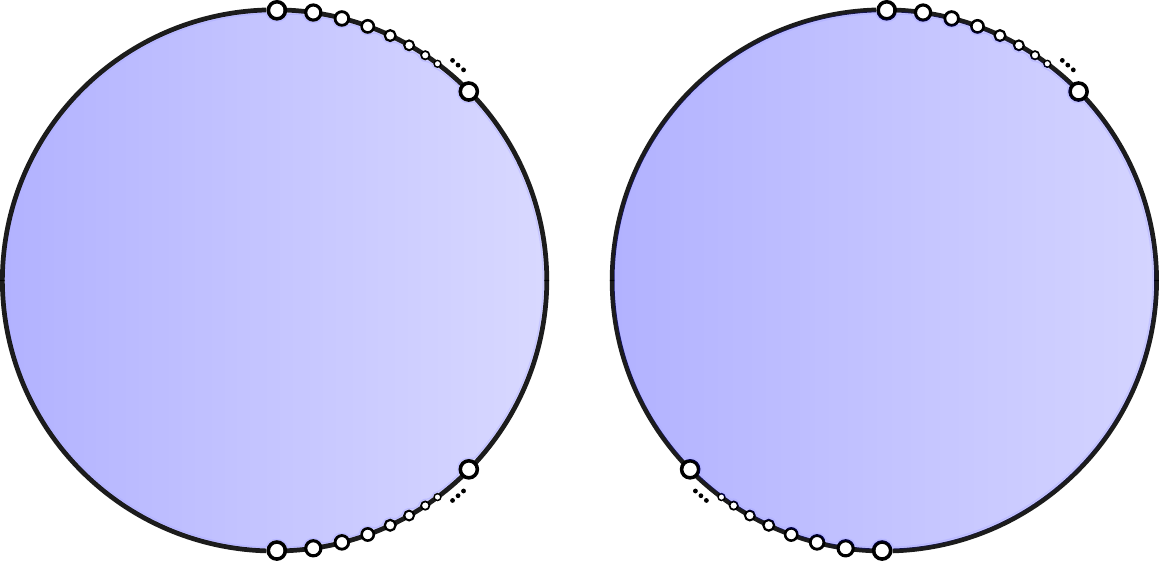}
    \caption{Nonhomeomorphic but proper homotopy equivalent surfaces.}
    \label{fig:nonhomeo2}
\end{figure}
 See \Cref{fig:nonhomeo2} for an example of two disks with boundary points removed. Although they have the same end spaces, they are not homeomorphic as the one has a noncompact boundary component whose ends are both accumulating points of other ends, but the other does not have such a noncompact boundary component. Hence, to distinguish them, we need more information. 
\end{example} 

We refer to the ends from \Cref{exa:diskminusboundarypoints} as \textit{degenerate ends}. More precisely, a \textbf{degenerate end} is an end that has a closed neighborhood $U$ which is homeomorphic to a disk with boundary points removed. 

\begin{example} [Sliced Loch Ness Monsters] Take a surface from
  \Cref{exa:diskminusboundarypoints} with $n \in\{1,2,\dots\} \cup \{\infty\}$
  noncompact boundary components. Attach infinitely many handles going out to
  the single end in the $n=1$ case, and in all other cases attach infinitely
  many handles in a manner that joins the ends into a single end. See
  \Cref{fig:SLMNs}. The resulting surface is the \textbf{$n$-Sliced Loch Ness
    Monster}. Note that when there are infinitely many ends to begin with
  $(n=\infty)$, the handles attached collapse all topology on the end space and
  we get homeomorphic surfaces (the $\infty$-Sliced Loch Ness Monster) no matter
  what infinite subset $I$ of the Cantor set we removed at the beginning; more
  precisely, this is due to 
  the classification theorem \Cref{thm:classification} and the following remark.
\end{example}

\begin{figure}[ht!]
    \centering
    \includegraphics[width=0.7\textwidth]{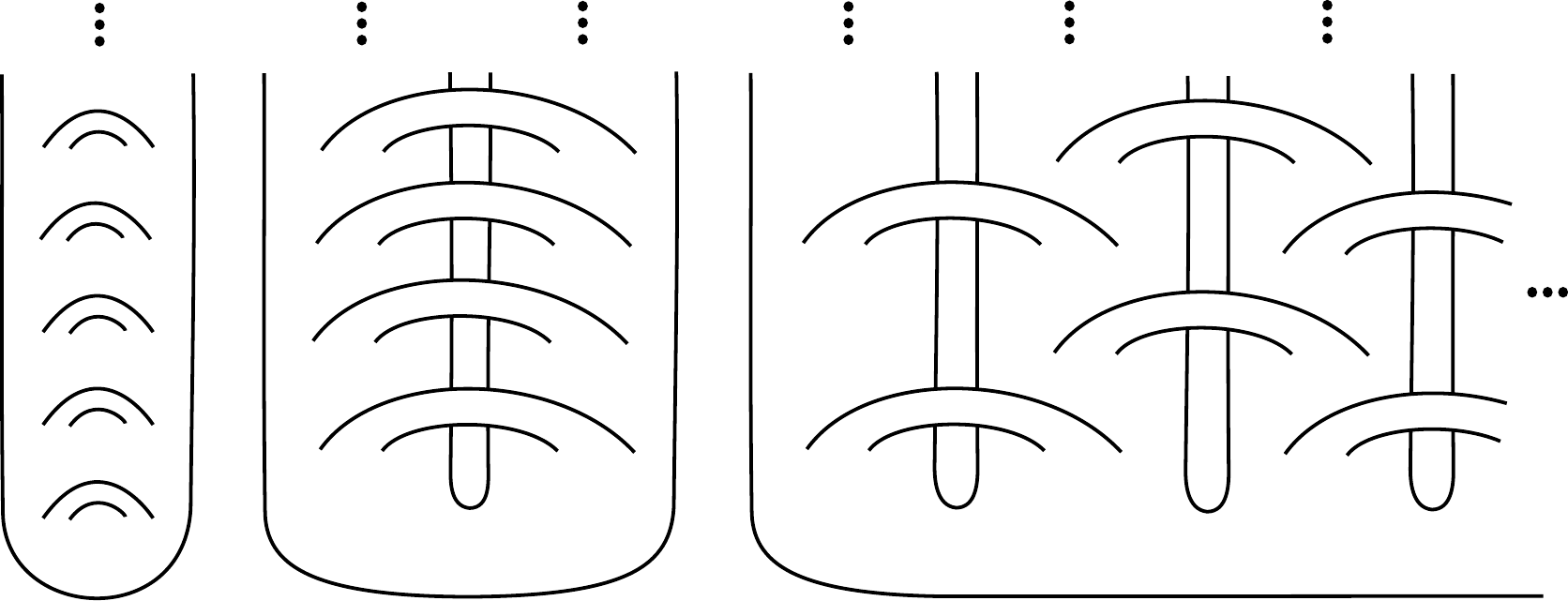}
    \caption{A 1-sliced, a 2-sliced, and an $\infty$-sliced Loch Ness monster.} 
    \label{fig:SLMNs} 
\end{figure}

Now we summarize the ideas for classifying surfaces with noncompact boundary following the work of Brown and Messer \cite{classification}. The previous definitions in \Cref{ssec:prelim_compactboundary} all apply to a general surface without adaptation, but we need more information to capture all the new possibilities. Note compact (or more generally finite-type) exhaustions for a surface $S$ with noncompact boundary components must include subsurfaces whose boundary intersects the noncompact boundary components of $S$ in a union of intervals.

For a surface with infinitely many compact boundary components, we must record the ends accumulated by compact boundary components. We refer to these as \textit{ends accumulated by compact boundary}, and we denote the space of these ends by $E_\partial(S)$. This can be precisely defined in a similar manner to $E_\infty(S)$, the space of ends accumulated by genus. For surfaces with only compact boundary components, Barros and Tondra independently proved the following classification. Here a homeomorphism between triples $h: (A,B,C) \to (D,E,F)$ is a homeomorphism $A \to D$ that restricts to homeomorphisms $h|_B:B \to E$ and $h|_C:C \to F$.
\begin{theorem}[{\cite{barros1974,tondra1979}}]\label{thm:classificationCptBdryComps}
    Two surfaces with only compact boundary components are homeomorphic if and only if the triples $(E(S),E_\infty(S),E_\partial(S))$ are homeomorphic and they have the same genus and the same number of boundary components.
\end{theorem}

Now we consider a surface with a noncompact boundary component, namely a component homeomorphic to the real line $\R$.
Let $\hat{\partial}S$ be the disjoint union of the noncompact boundary components of a surface $S$. Let $E(\hat{\partial}S)$ be the set of ends of $\hat{\partial}S$. This is just a discrete space with two points associated to each component. Let $v: E(\hat{\partial}S) \to E(S)$ be the function that takes an end of a noncompact boundary component to the corresponding end of the surface. Note it is possible that both ends of a noncompact boundary component get mapped by $v$ to the same end of $S$, as is the case for the 1-Sliced Loch Ness Monster. (See the leftmost surface of \Cref{fig:SLMNs}.)

Denote by $\pi_0(\hat{\partial}S)$ the discrete set of noncompact boundary components of $S$.
Let $e: E(\hat{\partial}S) \to \pi_0(\hat{\partial}S)$ be the map that takes an end of a noncompact boundary component to the corresponding noncompact boundary component.  If we fix an orientation on $S$, then for an arbitrary component $p \in \pi_0(\hat{\partial}S)$ we may distinguish the right and left ends of $e^{-1}(p)$. An \textit{orientation} of $E(\hat{\partial}S)$ is the subset $\mathcal{O} \subset E(\hat{\partial}S)$ that contains exactly the \emph{right} ends for the given orientation. We can collect all of this information in the following diagram:

\begin{equation} \label{diag:surfacediagram}\tag{$\ast$}
\begin{tikzcd}
    \pi_0(\hat{\partial}S) & E(\hat{\partial}S) \lar[swap]{e} \rar{v} & E(S) & E_{\infty}(S) \lar[hook] \\
    & \mathcal{O} \uar[hook] & E_{\partial}(S) \uar[hook] &
\end{tikzcd}
\end{equation}
where the unlabeled hooked arrows are the inclusions. We will refer to this as the \textit{surface diagram} for the surface $S$. Refer to \Cref{fig:surfacediagram} for an example. Two surfaces in \Cref{fig:nonhomeo2} are not homeomorphic, as the one on the left has a noncompact boundary component whose ends are both limit points of the image of $v$, but the one on the right has none. This yields nonhomeomorphic surface diagrams for those two surfaces. Also, see \Cref{fig:nonhomeo} for an example of nonhomeomorphic surfaces distinguished by orientations on ends.

\begin{figure}[ht!]
    \centering
    \includegraphics[width=.8\textwidth]{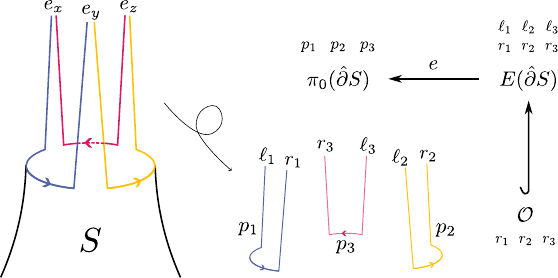}
    \caption{Illustration of a surface diagram. Here $p_1,p_2$ and $p_3$ are noncompact boundary components of $S$ with orientation inherited from $S$. With this orientation, each component $p_i \in \pi_0(\hat{\partial}S)$ has the left end $\ell_i$ and the right $r_i$, among which only $r_1,r_2$ and $r_3$ are contained in $\mathcal{O}$. Under the map $v:E(\hat{\partial}S) \to E(S)$, we have $v(r_3)=v(\ell_1) = e_x$, $v(r_1)=v(\ell_2)=e_y$, and $v(r_2)=v(\ell_3)=e_z$.}
    \label{fig:surfacediagram}
\end{figure}

See \cite{classification} for the construction of a surface from a given \textit{abstract surface diagram}, which is a diagram of the above form consisting of topological spaces and maps satisfying various technical conditions.
Here we say two diagrams are \emph{homeomorphic} if there exist homeomorphisms between the corresponding sets of the diagrams which commute with the arrows in the diagrams. We will not use abstract surface diagrams in this paper, so we leave it to the reader to review this definition if desired.

\begin{figure}[ht!]
    \centering
    \includegraphics[width=0.3\textwidth]{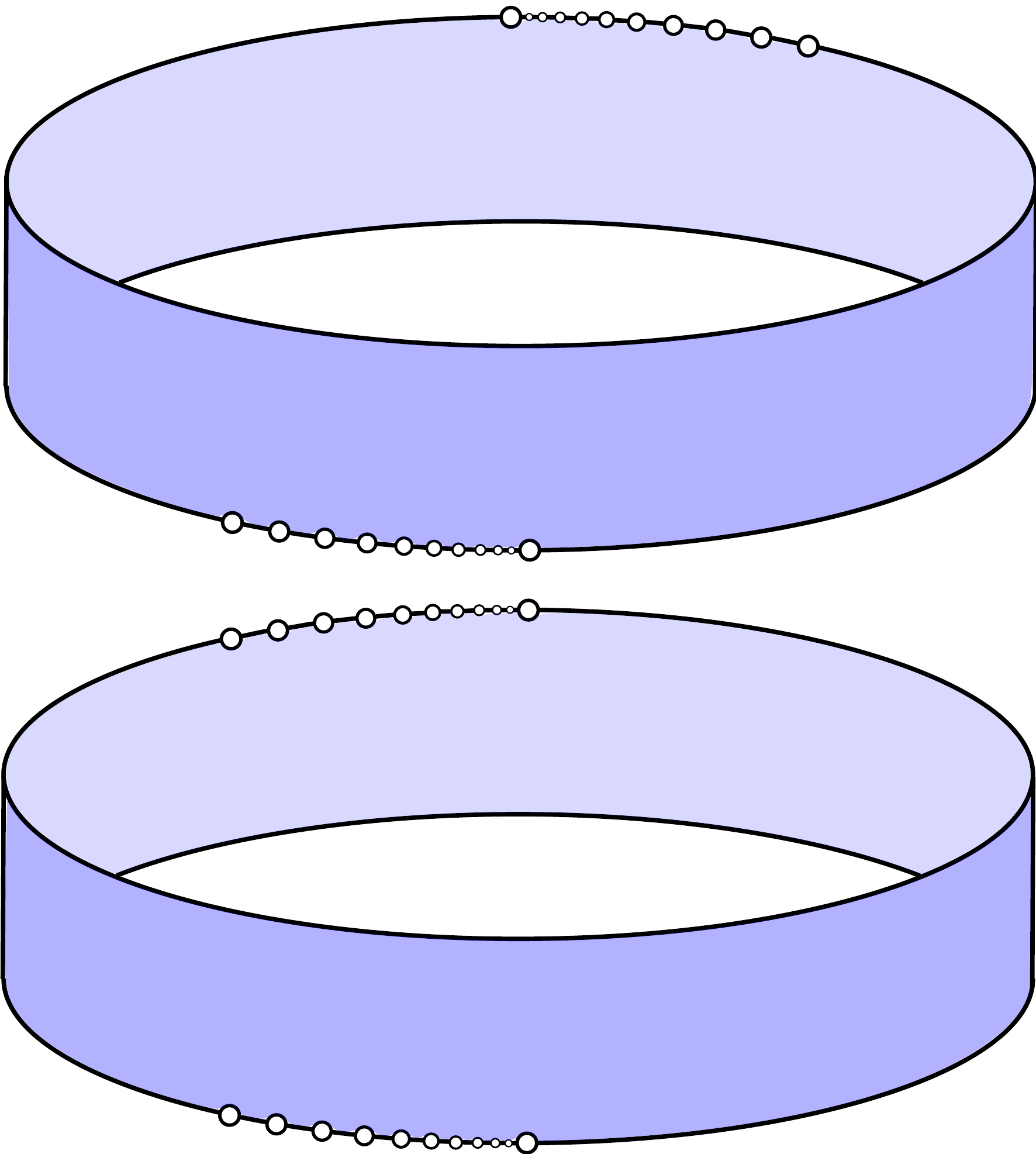}
    \caption{Nonhomeomorphic but properly homotopy equivalent surfaces.}
    \label{fig:nonhomeo}
\end{figure}

\begin{theorem} \label{thm:classification} (Classification of Surfaces, \cite{classification})
   Two surfaces are homeomorphic if and only if they have homeomorphic surface diagrams, the same genus, and the same number of compact boundary components.
\end{theorem}

 The classification theorem of Brown--Messer implies that any \emph{one-ended}
 surface of infinite genus that has only noncompact boundary components is
 determined by the number of the noncompact boundary components. Indeed, in this
 case we have $|E_\partial(S)|=0$, $|E(S)|=|E_\infty(S)|=1$, and the $v$ map in
 the surface diagram  \eqref{diag:surfacediagram} is a constant map. Any orientation $\mathcal{O}$ corresponds to a subset of $E(\hat{\partial} S)$ with one end from each $e^{-1}(x)$ where $x$ is a noncompact boundary component. Hence, $|E(\hat{\partial}S)| = 2|\pi_0(\hat{\partial}S)|$ determines the homeomorphism type of the diagram, and so does that of the surface by \Cref{thm:classification}. Perhaps surprisingly, this implies that the two surfaces from \Cref{fig:SLMNs2} are the 1-Sliced Loch Ness Monster and the $\infty$-Sliced Loch Ness Monster. This can be verified by counting their noncompact boundary components.

\begin{figure}[ht!]
    \centering
    \includegraphics[width=0.6\textwidth]{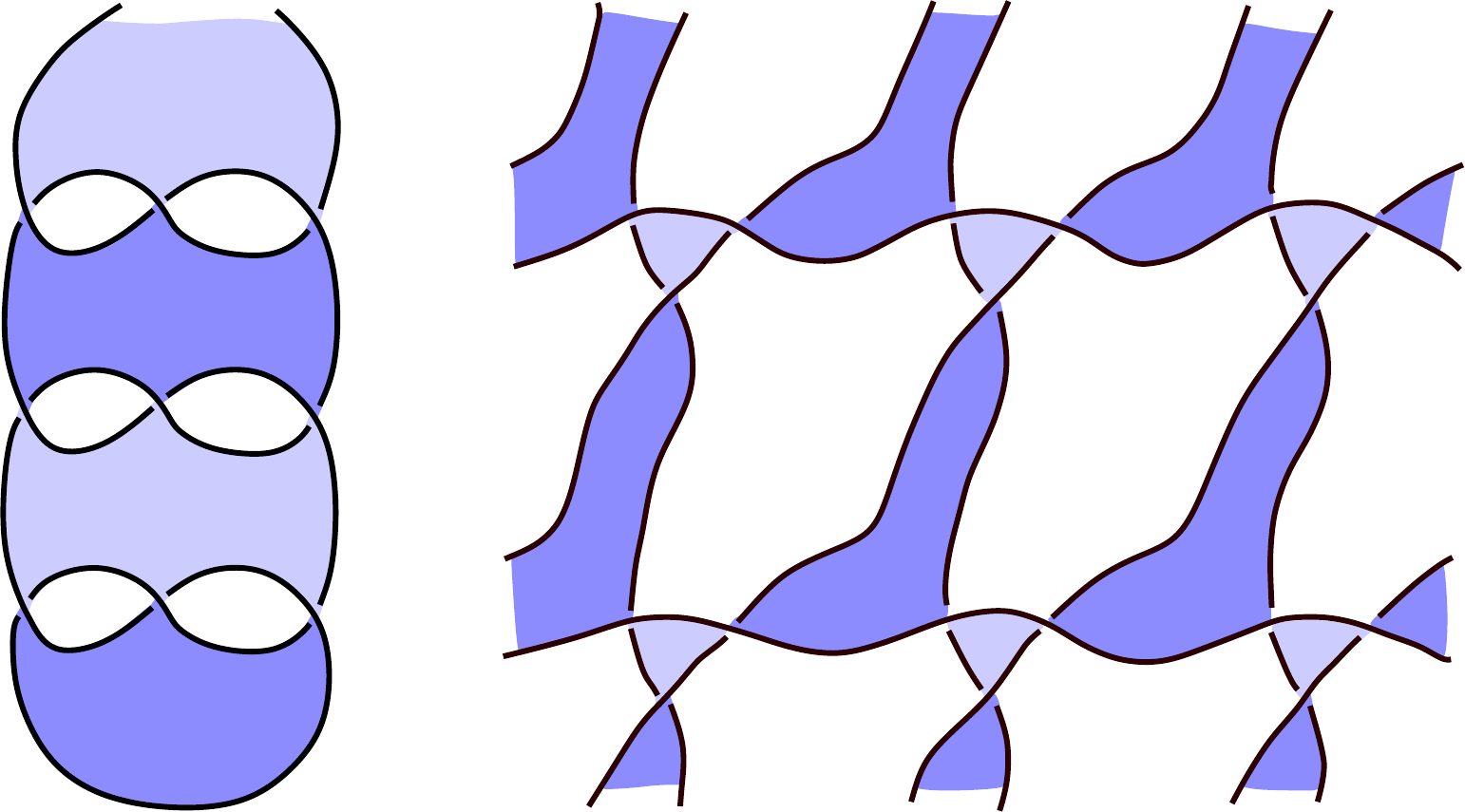}
    \caption{Embeddings of the 1-sliced Loch Ness monster (left) and the $\infty$-sliced Loch Ness monster (right) in $\R^3$ as Seifert surfaces for long knots/links. The right image was traced from Candel--Conlon \cite[Figure 12.5.11]{candel2000foliationsI}. Note these embeddings suggest a deformation retraction of each surface onto a spine.}
    \label{fig:SLMNs2} 
\end{figure}

Now we introduce a useful notion when decomposing a surface with noncompact boundary components.
Let $S^\circ = S \setminus \partial S$ be a surface obtained from $S$ by deleting the boundary. Then we get a canonical continuous map $\pi: E(S) \to E(S^\circ)$. See \cite[Section 4.2]{dickmann2023mapping} for a detailed construction of $\pi$. Let $V = v(E(\hat{\partial}S))$, the set of ends of $S$ coming from noncompact boundary components.

\begin{definition} (Boundary chains and boundary ends) \label{def:chains}
A \textbf{boundary chain} of a surface $S$ is a subset of $E(S)$ of the form $\pi^{-1}(p)$ for some $p \in \pi(V)$. A \textbf{boundary end} will refer to any end in a boundary chain. Occasionally, we will conflate definitions and use boundary chain to refer to the union of noncompact boundary components with ends in the chain. See \Cref{fig:boundaryend} for an example.
\end{definition}
\begin{figure}[ht!]
    \centering
    \includegraphics[width=0.7\textwidth]{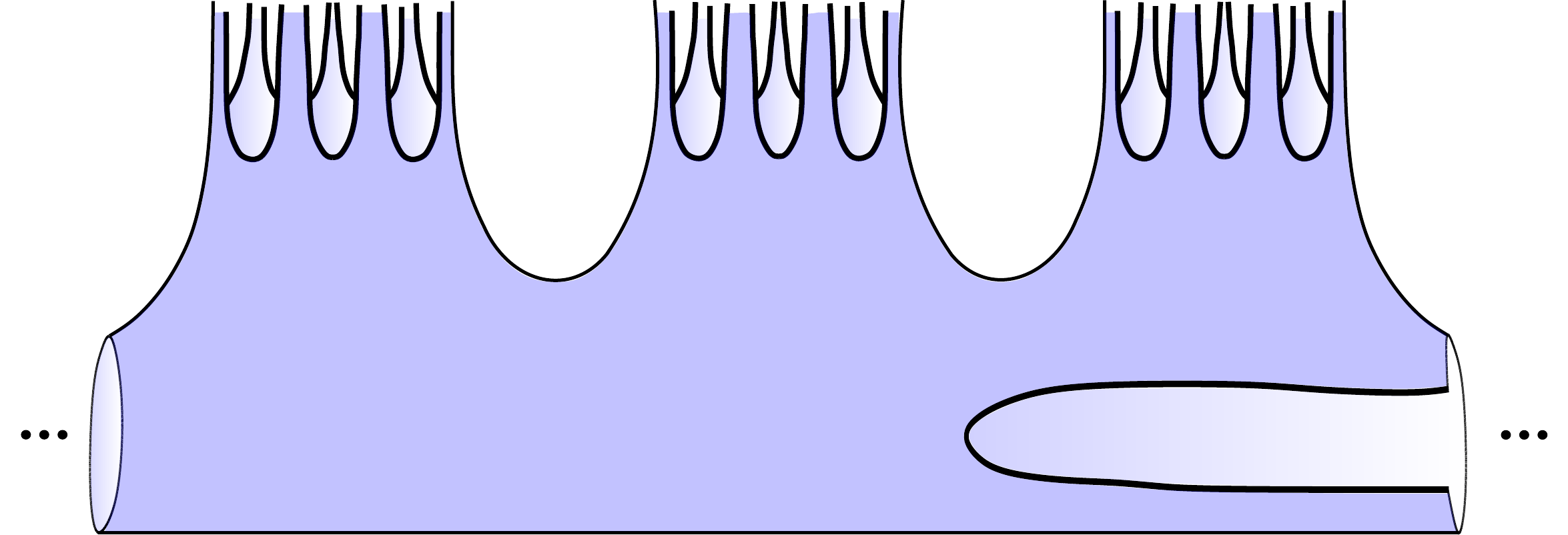}
    \caption{The ``leftmost" end of the surface shown is not a boundary end, even though it is accumulated by boundary chains. The ``rightmost" end is a boundary end.}
    \label{fig:boundaryend}
\end{figure}

 To build intuition, consider a compact surface $S$ with one compact boundary component. Form a surface $\hat{S}$ with noncompact boundary by removing a finite set of points $\mathcal{P}=\{p_1,\ldots,p_n\}$ from the boundary component of $S$. Each point $p_i$ corresponds to a degenerate end of the surface $\hat{S}$. If we delete the remaining boundary from $\hat{S}$ to obtain $S^\lilo$, then the original boundary component of $S$ becomes a single end. Because all the degenerate ends in $\hat{S}$ collapse to a single end when we delete the boundary, they form a boundary chain of $\hat{S}$.  This example did not rely on the set of points being finite at all; the same reasoning holds if we had instead deleted any closed subset of a Cantor set from the original boundary component of $S$.   

Also, we remark that a degenerate end is a boundary end but not vice versa; for example, a boundary end may be accumulated by genus.

\section{Surfaces not properly homotopic to graphs}
\label{sec:surfacesNotPHEtographs}

Now we prove the forward direction of \Cref{thm:maincharacterization}. For comparison, we will first use compactly supported cohomology to prove a simpler statement in \Cref{prop:notgraph_borderless}: a surface without boundary is not PHE to a graph. We then use a tailor-made variant of compactly supported cohomology to prove the general case in \Cref{prop:notgraph}: a surface is PHE to a graph only if it does not have a \emph{boundary isolating curve}. The idea to use compactly supported cohomology and some possible variant was suggested to us by Mladen Bestvina.

A curve on a surface $S$ is the image of an immersion $S^1 \to S$, and we say a curve is simple if it can be realized as the image of an embedding $S^1 \to S$. 
\begin{definition}
    Let $S$ be a surface with nonempty boundary. Then a simple closed curve $\alpha$ on $S$ is said to be \textbf{boundary isolating} if $S\setminus \alpha$ has two components $Y,Z$ with the following properties: \begin{itemize}
        \item $Y$ has a boundary component,
        \item $Z$ has no boundary, and
        \item $Z\cup \alpha$ is not compact.
    \end{itemize}
Equivalently, the quotient $S / \alpha$ is a wedge of surfaces consisting of a surface with nonempty boundary and a noncompact surface without boundary.
\end{definition}

See \Cref{fig:boundaryisolating} for examples. Also note a curve bounding a single puncture is a boundary isolating curve. The existence of a boundary isolating curve says that the surface ``has a large enough subsurface without boundary." On the other hand, for a surface without a boundary isolating curve we can say ``every end sees boundary." The following statement makes this precise.

\begin{proposition} \label{prop:isolatingnbhd}
    $S$ has a boundary isolating curve if and only if $S$ has nonempty boundary and there exists a closed neighborhood $U\subset S$ of some end such that $\partial U$ has one compact component.
\end{proposition}

\begin{proof}
    For the forward direction, take $U = \overline{Z}$. For the other direction, $\partial U$ is a boundary isolating curve.
\end{proof}

The following two facts are the main technical tools.

\begin{fact}[Poincar{\'e} Duality] \label{fact:PD}
    Let $M$ be a closed oriented (possibly noncompact) $n$-manifold and
    $k$ be any integer such that $0 \le k \le n$.
    Then we have an isomorphism
    \[
        H^k_c(M; \Z) \overset{\cong}{\longrightarrow} H_{n-k}(M; \Z).
    \]
\end{fact}

\begin{fact}[{\cite[Proposition III.6.6]{Ive1986}}, PHE-invariance of Compactly Supported Cohomology]
\label{fact:PHEinvariance}
Let $G$ be an abelian group.
Two properly homotopic maps $f,g:X \to Y$ between locally compact spaces $X,Y$ induce the same map
\[
f^*,g^*: H_c^*(Y,G) \to H_c^*(X,G).
\]
In particular, if $h:X \to Y$ is a proper homotopy equivalence (with proper homotopy inverse), then
\[
    h^*:H_c^*(Y,G) \to H_c^*(X,G)
\]is an isomorphism.
\end{fact}

These facts quickly show that a surface that is PHE to a graph must have nonempty boundary:
\begin{proposition} \label{prop:notgraph_borderless}
    If $S$ is a surface without boundary, then $S$ is not properly homotopy equivalent to a graph.
\end{proposition} 

\begin{proof}
 Poincar\'e duality (\Cref{fact:PD}) gives an isomorphism $H^2_c(S; \Z)\cong H_0(S; \Z)$, and $H_0(S; \Z)=\Z$ because we have assumed our surfaces are connected. On the other hand, $H^2_c(\G; \Z)=0$ for any graph $\G$ because it is one-dimensional. By \Cref{fact:PHEinvariance}, we see that $S$ cannot be properly homotopy equivalent to a graph.
\end{proof}

Note that we needed to use compactly supported cohomology as opposed to standard cohomology since the top cohomology group of a noncompact manifold is trivial. On the other hand, any manifold with boundary has trivial compactly supported top cohomology, so we need a new invariant to detect when an end is accumulated by boundary.

Let $n \ge 0$. Let $X$ be a locally compact, $\sigma$-compact, Hausdorff space. For compact sets $K,L \subset X$, an inclusion $K \hookrightarrow L$ induces a restriction homomorphism between relative compactly supported cohomology groups: $H^n_c(X,L) \to H^n_c(X,K)$. Let $H^\bullet_{c_\infty}(X; \Z)$ be the \textbf{compactly supported cohomology at infinity} of $X$, defined as the inverse limit:

\[
H^n_{c_\infty}(X; \Z) := \varprojlim_{K \subset X \: \text{compact}} H^n_{c}(X, K; \Z),
\]
for every $n \ge 0$. As a compact exhaustion $\{K_i\}_{i=1}^\infty$ of $X$ is cofinal with respect to the set of compact subsets of $X$, it can be simply put as:
\[
H^n_{c_\infty}(X; \Z) = \varprojlim_{i} H^n_{c}(X, K_i; \Z).
\]

We denote by $r_i^X: H^n_{c}(X, K_i; Z) \to H^n_{c}(X, K_{i-1}; \Z)$ the restriction homomorphisms induced by the inclusions $K_{i-1} \hookrightarrow K_i$ from the compact exhaustion of $X$.

\begin{lemma}\label{lem:cohomPHEinvariance}
Compactly supported cohomology at infinity is a proper homotopy invariant in locally compact, $\sigma$-compact, Hausdorff spaces. 
\end{lemma}

\begin{proof}
    Let $f: X \to Y$ be a proper map between locally compact, $\sigma$-compact, Hausdorff spaces. Take a compact exhaustion $\{K_i\}_{i=1}^\infty$ of $Y$. As $f^{-1}(K_i)$ is compact for each $i \ge 0$, the collection $\{f^{-1}(K_i)\}_{i=1}^\infty$ is a compact exhaustion for $X$. Also, $f$ restricts to a proper map $f_i:f^{-1}(K_i) \to K_i$ for each $i \ge 0$. Hence, it induces a homomorphism $(f_i)^*: H^n_c(Y,K_i;\Z) \to H^n_c(X,f^{-1}(K_i);\Z)$ for each $n \ge 0$ and $i \ge 0$.
    Then the $(f_i)^*$ commute with the restriction homomorphisms $r_i^X$ and $r_i^Y$ induced by the compact exhaustions of $X$ and $Y$ respectively, it follows by the universal property of inverse limits that $f^*:=\varprojlim_{i} (f_i)^*: H^n_{c_\infty}(Y;\Z) \to H^n_{c_\infty}(X;\Z)$ is a homomorphism.
    Moreover, once again by the universal property of inverse limits, for proper maps $f: X \to Y$ and $g:Y \to Z$, it follows that $(gf)^* = f^* g^* : H^n_{c_\infty}(Z;\Z) \to H^n_{c_\infty}(X;\Z)$.

    Consider the standard projection $p:X \times [0,1] \to X$ defined as $p(x,t):=x$ for each $x \in X$ and $t \in [0,1]$. We claim the induced map
    \[
        p^*: H^n_{c_\infty}(X;\Z) \to H^n_{c_\infty}(X \times [0,1];\Z)
    \]
    is an isomorphism. Consider a compact exhaustion $\{K_i\}_{i=1}^\infty$ of $X$. Then $p^{-1}(K_i) = K_i \times [0,1]$ for each $i \ge 1$, and $\{K_i \times [0,1]\}_{i=1}^\infty$ forms a compact exhaustion of $X \times [0,1]$. The restriction $p_i:=p|_{K_i \times [0,1]}:K_i \times [0,1] \to K_i$ is a proper map, as $p^{-1}(K_i) = K_i \times [0,1]$. Moreover, $p_i$ is a proper homotopy equivalence with a proper homotopy inverse (the inclusion). Therefore, $p$ induces an isomorphism $H^n_c(X,K_i;\Z) \to H^n_c(X \times [0,1],K_i \times [0,1];\Z)$ for each $i \ge 1$. Since the $r_i^X$ and $r_i^{X \times [0,1]}$ commute with the $p_i^*$, we get the desired isomorphism $p^*: H^n_{c_\infty}(X;\Z) \to H^n_{c_\infty}(X \times [0,1];\Z)$.

    Now suppose $f$ and $g$ are properly homotopic proper maps. We will show $f^*=g^*:H^n_{c_{\infty}}(Y;\Z) \to H^n_{c_\infty}(X;\Z)$ for $n \ge 0$.
    Let $H: X \times I \to Y$ be the homotopy between $f$ and $g$ such that $H(x,0) = f(x)$ and $H(x,1)=g(x)$ for all $x \in X$. For $t \in [0,1]$, let $\iota_t: X \to X \times [0,1]$ be the section of $H$ such that $\iota_t(x):=(x,t)$ for $x \in X$. 
    Then we have $p\iota_t = \id_X$ for every $t \in [0,1]$. 
    Also, by the previous paragraph, the induced map $p^*:H^n_{c_\infty}(X;\Z) \to H^n_{c_\infty}(X \times [0,1];\Z)$ is an isomorphism.
    Since $(p\iota_t)^* = (\id_X)^*=\id_{H^n_{c_\infty}(X)}$, we have $(\iota_t)^* = (p^*)^{-1}:H^n_{c_\infty}(X \times [0,1] ; \Z) \to H^n_{c_\infty}(X;\Z)$ for every $t \in [0,1]$. In particular, $(\iota_0)^*=(\iota_1)^*$. Since $f = H\iota_0$ and $g = H\iota_1$, it follows that
    \[
        f^* = \iota_0^* H^* = \iota_1^* H^* = g^*,
    \]
    as desired. 

    Finally, suppose $f:X \to Y$ is a proper homotopy equivalence and say $g: Y \to X$ is its proper homotopy inverse. That is, $f \circ g$ is properly homotopic to the identity, and so is $g \circ f$. By the previous paragraph, $(gf)^* = f^* g^*$ is the identity map on $H^n_{c_\infty}(X;\Z)$ and $g^* f^*$ is the identity map on $H^n_{c_\infty}(Y;\Z)$. This implies that $f^*:H^n_{c_\infty}(Y;\Z) \to H^n_{c_\infty}(X;\Z)$ is an isomorphism, concluding the proof.
\end{proof}

Now we are ready to prove the main proposition of this section.

\begin{proposition} \label{prop:notgraph}
    A surface has a boundary isolating curve if and only if it has nontrivial 2nd compactly supported cohomology at infinity. Therefore, a surface is PHE to a graph only if it has no boundary isolating curve.
\end{proposition}

\begin{proof}
The therefore statement follows from \Cref{lem:cohomPHEinvariance} and the observation that a graph has no 2-cells and thus trivial 2nd compactly supported cohomology at infinity. 

Suppose $S$ has a boundary isolating curve $\alpha$. Let $\{K_i\}$ be a compact exhaustion of $S$ such that each $K_i$ contains $\alpha$. Since $\alpha$ is boundary isolating, the quotient $S/ K_i$ is a wedge of surfaces at least one of which is noncompact with empty boundary. Fixing some $i$, we have from excision that $H^2_{c}(S, K_i; \Z) \cong H^2_c(S/K_i;\Z)$ with some nonzero $a \in H^2_c(S,K_i)$ corresponding to this noncompact surface with empty boundary.
A surface component in the wedge $S/K_i$ corresponds to a connected component of $S \setminus K_i$. Let $S_i$ be the connected component of $S \setminus K_i$ corresponding to $a$.

 Note for all $j > i$, the containment $K_i \subseteq K_j$ gives a reverse containment on connected components of the complements $S \setminus K_i$ and $S \setminus K_j$.
 Since $S_i$ has no boundaries or isolated ends (other than the ones coming from removing $K_i$ from $S$), there is some component $S_j$ of $S \setminus K_j$ without boundaries or isolated ends (excluding those arising from deleting $K_j$ from $S$), such that $S_j \subset S_i$. This implies 
 we get an $a' \in H^2_c(S,K_j)$ corresponding to $S_j$ whose image in $H^2_c(S,K_i)$ is $a$. Thus, we have that $H^2_{c_\infty}(S; \Z)$ is nontrivial.

On the other hand, if a surface does not have a boundary isolating curve then there exists $N$ such that for all $i \ge N$, either $S/ K_i$ is a wedge of surfaces with nonempty boundary or $K_i = S$ (when $S$ is compact)  so that $H^2_{c_\infty}(S; \Z)$ is trivial. \end{proof}

We note that the above proof also applies to \Cref{prop:notgraph_borderless}. The 2nd compactly supported cohomology at infinity is also useful for distinguishing the proper homotopy type of some surfaces.

\begin{example}
    Let $S_1$ be the genus zero surface with three ends, two of which are accumulated by compact boundary components. Let $S_2$ be the genus zero surface with three ends, one accumulated by compact boundary components. By \Cref{prop:notgraph}, these two surfaces are not PHE to a graph. Moreover, we can distinguish their PHE types using $H^2_{c_\infty}$. Since $H^2_{c_\infty}(S_1; \Z) = \Z$ and $H^2_{c_\infty}(S_2; \Z) = \Z^2$, these surfaces are not PHE to each other.
\end{example}

We note that there are some alternative approaches to the problem that may be useful when considering the proper homotopy classes of all surfaces. One can directly prove a surface with an interior puncture is not PHE to a graph using the fact that any finite tuple of non-nullhomotopic curves that are homotopic intersects. First, parameterize the puncture as $S^1 \times [0,\infty)$ and note the images of the $S^1 \times \{t\}$ are not nullhomotopic as a PHE preserves $\pi_1$. Since they are homotopic to each other in $S$, their images are also homotopic in the graph. In particular, the images of $S^1 \times \{0\}$ and $S^1 \times \{t\}$ intersect for all $t>0$, so the inverse image of the image of $S^1 \times \{0\}$ is noncompact, contradicting the properness.

Alternatively, surfaces with any interior planar ends are not PHE to a graph by
the following line of reasoning: 1) proper homotopy equivalent spaces
have homotopy equivalent end space compactifications, 2) a graph and its
end space compactification are homotopy equivalent by the natural inclusion, and 3) the natural inclusion of a surface to its compactification is \emph{not} a homotopy equivalence since loops bounding interior punctures are killed in the compactification.
We wonder if these ideas and compactly supported cohomology at infinity can be used to address the following problem.

\begin{problem}
Classify surfaces up to proper homotopy equivalences.
\end{problem}

On the other hand, one can consider the following proper homotopy invariant. Brown extended Whitehead's theorem to the proper category \cite{brownphe} showing that two connected locally finite simplicial complexes $X, Y$ are in the same proper homotopy class if and only if there exists a proper map $f: X \rightarrow Y$ such that $f$ induces a homeomorphism on the ends spaces, homotopy groups, and the \textit{proper homotopy groups}.
 
Let $\hat{S}^n$ be the space defined by attaching a distinct $n$-sphere at each integer point of the half-line $[0, \infty)$. Let $\rho$ denote the half-line in $\hat{S}^n$. For a space $X$ and a properly immersed ray $\alpha$ in $X$, by a proper map of pairs $(\hat{S}^n, \rho) \rightarrow (X, \alpha)$ we mean a proper map $\hat{S}^n \rightarrow X$ where the image of $\rho$ agrees with $\alpha$ outside of a compact set.

\begin{definition}
The \textbf{$n^{th}$ proper homotopy group} $\pi^\star_n(X, \alpha)$ of a space $X$ based at a properly immersed ray $\alpha$, consists of equivalence classes of proper maps of pairs $(\hat{S}^n, \rho)$ to $(X,\alpha)$ where two maps of pairs are equivalent when they have representatives that are properly homotopic relative to $\rho$. Composition of elements is applied by considering representatives that agree on $\rho$ and composing each pair of $n$-spheres that meet at an attaching point via the usual composition in $\pi_n$. 
\end{definition}

For any two rays $\alpha, \beta$ that go to the same end of $X$, Brown showed $\pi^\star_n(X, \alpha)$ is naturally isomorphic to $\pi^\star_n(X, \beta)$ so that we may consider the proper homotopy group to depend on a choice of \textit{end} $e$ which we denote by $\pi^\star_n(X, e)$.

We remark that the PHE classification of graphs in \cite{ayala1990proper} follows from Brown's theorem. Our first useful observation from this theorem in our context is the following.

\begin{proposition}

The standard inclusion of a sliced Loch Ness monster $L_s$ into the Loch Ness monster $L$ given by attaching closed upper half-planes onto each boundary component does not induce a surjection on first proper homotopy groups. Therefore, the inclusion is not a PHE.
    
\end{proposition}

\begin{proof}
   Note the inclusion is a proper map and a homotopy equivalence, but the first statement of the lemma will imply it is not a PHE by Brown's theorem. Fix a properly immersed ray $\alpha$ lying within $L_s$ so we have a map $\pi^\star_1(L_s, \alpha)$ to $\pi^\star_1(L, \alpha)$ induced by inclusion.

     Consider a compact exhaustion of $L$ where each compact subsurface has one boundary component. Let $\{\beta_i\}$ denote the sequence of boundary components viewed as separating curves in $L$. Infinitely many of the $\beta_i$ intersect $\alpha$, so we can attach the $\beta_i$ to  $\alpha$ to get an element of $\gamma \in \pi^\star_1(L, \alpha)$. Note we cannot properly homotope $L$ so that all the $\beta_i$ simultaneously lie within $L_s$. Therefore, there is no element of $\pi^\star_1(L_s, \alpha)$ that maps to $\gamma$.
\end{proof}

Although the above proposition also follows from \Cref{thm:maincharacterization}, its proof can be modified to show the following statement that does not follow immediately from \Cref{thm:maincharacterization}.

\begin{proposition}
Let $m>n$ be finite positive integers, or $m=\infty$.
The inclusion of an $m$-sliced Loch Ness monster into an $n$-sliced Loch Ness monster given by capping some of the boundary components with upper half-planes is not a PHE.
\end{proposition}

\section{Surfaces properly homotopy equivalent to graphs}
\label{sec:surfacesPHEtograph}

In this section, we prove the reverse direction of \Cref{thm:maincharacterization}:

\begin{proposition} \label{prop:reverse}
    A surface with nonempty boundary and no boundary isolating curve is PHE to a graph.
\end{proposition}

Our strategy is as follows. Given a surface with no boundary isolating curve, we will\begin{enumerate}
    \item cut the surface into compact pieces via some process, 
    \item deformation retract each piece onto a graph, fixing the cutting curves and arcs, then
    \item glue the pieces back together along the cutting curves and arcs to get a (proper) deformation retract of the entire surface onto a graph.
    \end{enumerate}

    Throughout we will use the term \textit{original boundary} to refer to boundary of a subsurface that is also part of the boundary of the whole surface. For step (2), we will require that each compact piece contains original boundary. This original boundary will be allowed to move during the deformation retraction while the boundary components from the cutting must stay fixed.
    
    It is standard knowledge that a compact surface with nonempty boundary deformation retracts onto a graph, however, we want to emphasize that this deformation retraction can be done in a way that fixes any proper closed subset, so we give the following statement whose proof we leave to the reader. 
     \begin{lemma}\label{lem:compactdefretraction}
Let $S$ be a compact surface with nonempty boundary, and let $I \subsetneq \partial S$ be a closed subset. Then there exists a proper deformation retraction of $S$ to a graph which fixes $I$. 
\end{lemma}

 Note that the proper containment $I \subsetneq \partial S$ is required in \Cref{lem:compactdefretraction}. If there were a deformation retraction of $S$ to a graph which fixes all of $\partial S$, then by collapsing the boundary components of $S$ to points we get a deformation retraction of a closed surface to a graph. This is impossible by \Cref{prop:notgraph_borderless}.

Now we present a proof of the reverse direction of \Cref{thm:maincharacterization} where the first step of our strategy is entirely non-constructive relying only on the condition that ``every end sees boundary".

\begin{proof}[Proof of \Cref{prop:reverse}]
     Suppose $S$ has nonempty boundary and no boundary isolating curves. Let $\{K_i\}_{i=1}^\infty$ be a compact exhaustion of $S$. By enlarging $K_i$'s, we may assume the complementary components of $K_i$ are infinite-type surfaces. We want to show that we can find a new exhaustion $\{K^\prime_i\}_{i=1}^\infty$ so that each component of $K^\prime_{i+1} \setminus K^\prime_i$ contains original boundary. Then by \Cref{lem:compactdefretraction} we can find a proper homotopy equivalence of the closure of each component to a graph that fixes its boundary except for some piece of original boundary. We are then done by gluing together all proper homotopy equivalences into a single proper homotopy equivalence on $S$.
    
    To find this new exhaustion, we begin by setting $K^\prime_1$ to be the first $K_i$ that contains original boundary. Note $S \setminus K^\prime_1$ has finitely many components $U_1^{(1)},\ldots,U_{n_1}^{(1)}$, which are all infinite-type by our choice of $K_i$, and each $U_j^{(1)}$ contains original boundary, otherwise there is a boundary isolating curve. Now we can set $K^\prime_2$ to be the first $K_i$ such that $K_i \cap U_j^{(1)}$ contains original boundary for all $j=1,\ldots,n_1$. Letting the complementary components of $S\setminus K_2'$ be $U_1^{(2)},\ldots,U_{n_2}^{(2)}$, choose $K_3'$ to be the first $K_i$ such that $K_i \cap U_j^{(2)}$ contains original boundary for $j=1,\ldots,n_2$. Doing this process inductively, we obtain the desired exhaustion $\{K_i'\}_{i=1}^\infty$. 
\end{proof}

\begin{remark}
See \Cref{fig:remark} for some basic examples. The deformation retraction has the following effects on ends. Ends accumulated by genus become ends accumulated by loops. A neighborhood of a degenerate end deformation retracts to a tree and thus the degenerate ends become ends not accumulated by loops. 

However, other planar ends of the surface may be mapped to ends accumulated by loops. The two possible cases are ends accumulated by compact boundary components and ends accumulated by boundary chains (see \Cref{fig:boundaryend}) which are both mapped to ends accumulated by loops.  
\end{remark}

\begin{figure}[ht!]
    \centering
    \includegraphics[width=0.8\textwidth]{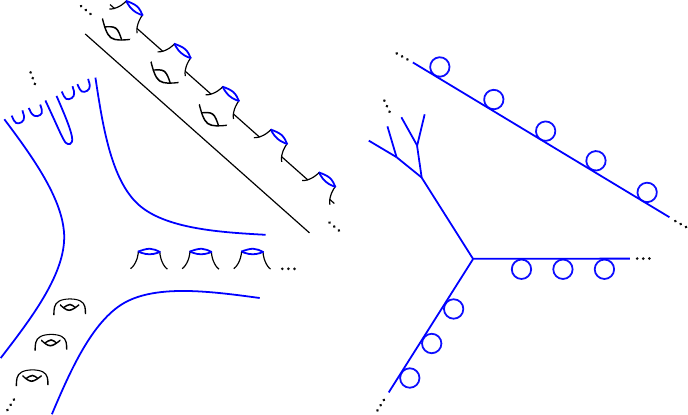}
    \caption{The two surfaces on the left (with boundary components shown in blue) are proper homotopy equivalent to the graphs on the right.}
    \label{fig:remark}
\end{figure}

The cut-and-paste construction of the deformation retraction in the proof above provides the following corollary which we use in the following sections.

\begin{corollary}\label{cor:defretraction_restrict} 
    Let $S$ be a surface with boundary, but without boundary isolating curve, and let $\phi: S \to \G$ be a proper deformation retraction to a locally finite, infinite graph properly homotopy equivalent obtained from \Cref{thm:maincharacterization}. Then there exist compact exhaustions $\{S_i\}$ for $S$ and $\{\G_i\}$ for $\G$ such that for each $i$, the restriction $\phi|_{S_i}:S_i \to \G_i$ is a proper deformation retraction.
\end{corollary}

\section{Dehn--Nielsen--Baer maps}
\label{sec:DNB} 
Consider the surfaces properly homotopy equivalent to graphs, which are characterized as the surfaces without boundary isolating curves by \Cref{thm:maincharacterization}. For such surface-graph pairs, we can establish a map between their mapping class groups as follows.

\subsection{The mapping class group}

Let $(S,\G)$ be a surface-graph pair such that there is a proper homotopy equivalence $\phi:S \to \G$ with a proper inverse $\psi: \G \to S$. Then we have a map
\[
    \Theta: \Map(S) \longrightarrow \Map(\G), \qquad [f] \mapsto [\phi f \psi],
\]
which is defined to complete the following commutative diagram given a homeomorphism $f: S \to S$:
\[
    \begin{tikzcd}
        \G \rar[dashed] \dar{\psi} & \G \dar{\psi} \\
        S \rar{f} & S
    \end{tikzcd}
\]

We now establish the basic properties of the map $\Theta$ as listed in \Cref{thm:MainTHM_DNB} proving the theorem over many propositions. Recall \Cref{prop:InducedMapsOnEnds} implies that a PHE $\phi:S \to \G$ induces a homeomorphism between the end spaces of $S$ and $\G$. 

Moreover, we have an induced map between their mapping class groups:

\begin{proposition} \label{prop:HomMappingClassGroups}
    The map $\Theta: \Map(S) \to \Map(\G)$ is a well-defined homomorphism.
    Moreover, the map $\Theta$ has image in $\PMap(\G)$. In other words, we can write $\Theta: \Map(S) \to \PMap(\G)$.
\end{proposition}

\begin{proof}
    First we show $\Theta$ is well defined. Suppose $f$ and $g$ are isotopic homeomorphisms on $S$. Then $[f]=[g] \in \Map(S)$. Since an isotopy is a proper homotopy, it follows that $\phi f \psi$ and $\phi g \psi$ are properly homotopic as PHEs on $\G$. Hence, $[\phi f \psi] = [\phi g \psi]$ in $\Map(\G)$, concluding the well-definedness.

    We obtain that $\Theta$ is a homomorphism by observing that for $[f],[g] \in \Map(S)$:
    \[
        \Theta([f])\Theta([g]) = [\phi f \psi \phi g \psi] = [\phi fg \psi] = \Theta([fg]),
    \]
    where in the second equality we used that $\psi \phi$ is properly homotopic to the identity.

    To show the moreover statement, recall $S$ has no boundary isolating curves, so every neighborhood of every end of $S$ has boundary. Because $[f] \in \Map(S)$ fixes boundary pointwise, it induces the identity map on $E(S)$. By functoriality of $E$ (\Cref{prop:InducedMapsOnEnds}),  $\Theta([f])$ must induce the identity map on $E(\G)$ as well, i.e. $\Theta([f]) \in \PMap(\G)$.
\end{proof}

We note the following proposition which we use in the next section. 
\begin{proposition}\label{prop:conjugate}
    Let $\phi,\phi': S \to \G$ be proper homotopy equivalences, and let $\phi_*,\phi'_*:\Map(S) \to \PMap(\G)$ be the induced homomorphisms. Then $\phi_*(\Map(S))$ and $\phi'_*(\Map(S))$ are conjugate in $\PMap(\G)$.
\end{proposition}

\begin{proof}
Fix $\psi$ and $\psi'$ proper homotopy inverses of $\phi$ and $\phi'$ respectively.
Then for any $[f] \in \Map(S)$,
\[
    \phi_*([f]) = [\phi f \psi] = [\phi \psi' \left(\phi' f \psi'\right) \phi'  \psi] = [\phi\psi'] \cdot \phi'_*([f]) \cdot [\phi' \psi],
\]
where we observe $[\phi\psi']^{-1} =[\phi'\psi]$ in $\PMap(\G)$.
\end{proof}

Our next goal is to understand the kernel of $\Theta$. We start with a warm-up example with a finite-type surface with compact boundary which demonstrates our \Cref{thm:MainTHM_DNB} generalizes the classical Dehn--Nielsen--Baer Theorem (\Cref{thm:classic_DNB}). 
\begin{example}[Dehn--Nielsen--Baer Theorem for finite-type $S$ with $\partial S\neq \emptyset$]
\label{exa:DNBCompactI}
   Consider a genus $g$ surface $S$ with $b>0$ boundary components. Let $* \in \partial S$ be a basepoint for $\pi_1(S, *) \cong F_n$. Since mapping class group representatives fix the boundary pointwise, for a given element \( [f] \in \mathrm{Map}(S) \) we can define the automorphism $f_\star  \in \Aut(F_n)$ by
$f_\star : \alpha \mapsto  f(\alpha)$,
 and let $[f_\star]$ be the corresponding outer automorphism. The Dehn--Nielsen--Baer map $\Delta:\Map(S) \to \Out(F_n)$ sends $[f]$ to $[f_{\star}]$.   
   
   To see this in the context of our DNB map induced from proper homotopy equivalence, let $\psi: R_n \hookrightarrow S$ be an inclusion of a rose graph $R_n$ to $S$, and let $\phi:S \to R_n$ be a deformation retraction. For simplicity, assume $* \in R_n$. Let $\phi_\star$ be the induced map $\pi_1(S, *) \to \pi_1(R_n, *)$ and similarly for $\psi_\star$. By the choice of PHE and basepoint, we can conflate elements in $\pi_1(S, *)$ with elements in $\pi_1(R_n, *)$ or in $F_n \cong \pi_1(S, *)$. In particular, we can view $\phi_\star\psi_\star$ as the identity automorphism of $F_n$.
   
   Then as graphs are $K(G,1)$, it is well known we have the isomorphism $r:\Map(R_n) \cong \Out(F_n)$. To illustrate this, we define $g_\star \in \Aut(\pi_1(R_n, *)) \cong \Aut(F_n)$ for a given PHE $g: R_n \to R_n$ as $g_\star : \alpha \mapsto  \gamma g(\alpha)\gamma^{-1},$ where $\gamma$ is some path from $*$ to $g(*)$. This is unique when considered up to inner automorphism of $\pi_1(R_n,\ast)$, so let $[g_\star]$ be the corresponding unique outer automorphism. Now $r$ maps $[g]$ to $[g_\star]$.
    
   Let $\Theta$ be the map $\Map(S) \rightarrow \Map(R_n)$ induced by the (proper) homotopy equivalences $\phi$ and $\psi$.  Then we claim the composition $r \circ \Theta$ is in fact the Dehn--Nielsen--Baer map $\Delta:\Map(S) \to \Out(F_n)$ we have considered above. Indeed, for a mapping class $[f] \in \Map(S)$, we have the following sequence of equalities in $\Out(F_n)$: \[
    (r \circ \Theta)([f]) = r([\phi f \psi]) = [(\phi f \psi)_\star] = [\phi_\star f_\star \psi_\star] = [f_\star] = \Delta([f]),
\]
where the second to the last equality follows from $\psi_\star = \phi_\star^{-1}$ in $\Aut(F_n)$ and conjugating the inside of $[\phi_\star f_\star \psi_\star]$ by $\phi_\star$.

 The map $\Delta$ factors through the similarly defined Dehn--Nielsen--Baer map from \Cref{thm:classic_DNB}, so the kernel is also given by the subgroup generated by Dehn twists about the boundary. This agrees with our \Cref{thm:MainTHM_DNB} since $r \circ \Theta = \Delta$.\end{example}

The main tools for the remainder of the section are based on the often used Alexander methods for surfaces (see \cite{hern2017isomorphisms,Shapiro2022Alexander, dickmann2023mapping} for information on the infinite-type case), and Alexander methods for graphs by Algom-Kfir--Bestvina as in the following lemma.

Recall a (closed) curve is an immersed image of $S^1$, an arc in a surface is a properly immersed image of $[0,1]$ with endpoint on the boundary of the surface, and a line is a properly immersed image of $\R$. 
\begin{lemma} [{\cite[Corollary 3.6]{AB2025}}] \label{lem:AB_Alexander}
    Let $f : \Gamma \rightarrow \Gamma$ be proper. Then $f$ is properly homotopic to the identity if and only if it preserves the homotopy class of every oriented closed curve and the proper homotopy class of every oriented proper line in $\Gamma$ that in each direction converges to an end in $E(\Gamma) \setminus E_\ell(\Gamma).$
\end{lemma}

We will also need a version for surfaces that relies on an Alexander system on a surface with noncompact boundary components.

\begin{lemma}[{See also \cite[Theorem 6.8]{dickmann2023mapping}}] \label{lem:twists_regular} 
Let $S$ be any surface with nonempty boundary. If $[f] \in \Map(S)$ fixes every oriented closed curve up to homotopy, it is in the subgroup topologically generated by Dehn twists about the compact boundary components and Dehn twists about the degenerate boundary chains. 
\end{lemma}

\begin{proof}

Note the Alexander method for surfaces without boundary from \cite{hern2017isomorphisms} shows that $[f]$, when considered as a homeomorphism on the interior, is the identity element for the mapping class group of the interior, but we need more to show the lemma. However, if we also fix all arcs up to homotopy (relative to the boundary), then the general Alexander method from \cite{dickmann2023mapping} shows $[f]$ is the identity. We will use this to show there is some representative $f$ that is the identity on a subsurface, then achieve the lemma by examining the mapping class group of the complementary surfaces.

   To simplify, fill in the degenerate ends of $S$ to turn each degenerate chain into a compact boundary component and call the resulting surface $S^{\deg}$. Note that $\Map(S)$ and $\Map(S^{\deg})$ are naturally isomorphic, so we conflate $[f]$ with an element of $\Map(S^{\deg})$.
   Since $f$ fixes every curve of $S^{\deg}$ up to homotopy and fixes the boundary pointwise, we have that $f$ must also fix every arc $\alpha$ up to homotopy possibly moving the endpoints. A homotopy that moves the endpoints may untwist an arc if it lies on a compact boundary component.
  Let $N$ be the union of open regular neighborhoods of the (compact) boundary components of $S^{\deg}$ so that $S^{\deg} \setminus N$ is homeomorphic to $S^{\deg}$. Note we can homotope $f$ so that it preserves components of $N$ and restricts to a homeomorphism on $S^{\deg} \setminus N$ fixing the boundary pointwise. Since we can homotope $f(\alpha)$ to $\alpha$ possibly moving the endpoints, we can compose this homotopy with an additional homotopy in $N$ to ensure that $f(\alpha)$ and $\alpha$ be homotopic as arcs in $S^{\deg} \setminus N$ by a homotopy fixing the boundary. Now since $f$ fixes every curve and arc of $S \setminus N$ up to homotopy, we apply the Alexander method for surfaces with boundary from \cite{dickmann2023mapping} in $S^{\deg} \setminus N$ to show $f$ is homotopic to the identity. Therefore, after homotopy, $f$ is supported in $N$. The components of $N$ are annuli, so it follows that $f$ can be written as an infinite product of twists in these annuli in $S^{\deg}$, so we are done. 
\end{proof}

We are now ready to prove the first part of \Cref{thm:MainTHM_DNB}.
\begin{theorem} \label{thm:kernel_regular}
     Let $S$ be a surface PHE to a graph $\G$. Then the kernel of $\Theta: \Map(S) \rightarrow \PMCG(\G)$ is topologically generated by Dehn twists about the compact boundary components. 
\end{theorem} 

\begin{proof}

    Suppose $[f]$ is in the kernel of $\Theta$. By \Cref{lem:AB_Alexander}, $\Theta([f])$ fixes the homotopy class of every oriented closed curve in $\G$. In order to apply \Cref{lem:twists_regular}, we want to show $f$ fixes the homotopy class of every oriented closed curve of $S$.
    Let $\phi : S \to \G$ be a proper homotopy equivalence and $\psi: \G \to S$ be a proper homotopy inverse used to define $\Theta$. Let $\alpha \subset S$ be an oriented closed curve in $S$. Then as $\Theta([f]) = [\phi f \psi]$, and $[f] \in \ker \Theta$, we have that
    \[
    (\phi f \psi) (\phi(\alpha)) \simeq \phi(\alpha).
    \]
    However we know $\psi \phi \simeq \id_{S}$, so $(\phi f)(\alpha) \simeq \phi(\alpha)$. Finally, by postcomposing $\psi$ on the equation we get $f(\alpha) \simeq \alpha$. Hence, $f$ preserves every oriented closed curve on $S$ up to homotopy, and thus by \Cref{lem:twists_regular}, $f$ is a possibly infinite product of Dehn twists about compact boundary components and degenerate chains.
    
    It remains to show the subgroup topologically generated by Dehn twists about compact boundary components lies in the kernel, but not the subgroup topologically generated by Dehn twists about degenerate chains. Each Dehn twist about a compact boundary component is in the kernel since its support can be homotoped to be disjoint from any curve or line in the surface, and so it fixes each curve and line up to homotopy in the graph. Such a homotopy can be done simultaneously even for infinitely many collection of Dehn twists about compact boundary components. Hence, any infinite product of these twists must be in the kernel by \Cref{lem:AB_Alexander}. Note whenever the curve $\gamma$ about the degenerate chain is nontrivial, any nontrivial power of the Dehn twist about $\gamma$ is mapped by $\Theta$ to a PHE that does \emph{not} fix every nontrivial line going to and from a degenerate end in the chain. Note after homotopy, such a nontrivial line still intersects $\gamma$, and powers of $T_\gamma$ increase the (minimal) intersection number with $\gamma$. Thus, \Cref{lem:AB_Alexander} implies the Dehn twist about a degenerate chain is not in the kernel.
\end{proof}

We now want to understand the interplay between topologies for mapping class groups. We equip $\Map(S)$ with the standard topology given by a neighborhood basis of the identity consisting of clopen subgroups $U_K$ of elements with representatives that pointwise fix a given compact subsurface $K$. See \cite{vlamis2019notes} for more details. The cosets of these subgroups then give a basis for all of $\Map(S)$.

Following \cite{AB2025}, we equip $\Map(\G)$ with a similar topology where the neighborhood basis of the identity is given by the clopen subgroups $U_\Delta$ of elements with representatives $f$ that satisfy the following properties:

\begin{enumerate}
    \item $f$ fixes a finite subgraph $\Delta$ pointwise, 
    \item $f$ preserves the complementary components of $\Delta$,
    \item $f$ has a proper homotopy inverse $g$ satisfying the first two properties, and
    \item $fg$ and $gf$ are homotopic to the identity via proper homotopies that are stationary on $\Delta$ and preserve the complementary components.
\end{enumerate}

The mapping class groups $\Map(S)$ and $\Map(\G)$ are known to be Polish groups with these topologies (See \cite[Corollary 5]{vlamis2019notes}, \cite[Proposition 4.11]{AB2025}.) Recall the kernel of $\Theta$ is topologically generated by Dehn twists about compact boundary components by \Cref{thm:kernel_regular}. Since the kernel is closed, the quotient of $\Map(S)$ by the kernel is a Polish group with the quotient topology, call it $\Map^\star(S)$. Let $\Map_S(\G)$ denote the image of $\Theta$ equipped with the subspace topology from $\Map(\G)$. Now $\Theta$ descends to an abstract group isomorphism $\Theta^\star$ from $\Map^\star(S)$ to $\Map_S(\G)$.

\begin{proposition} \label{prop: closedimage}
The map $\Theta$ is continuous. Furthermore, $\Theta^\star$ is a topological isomorphism and $\Map_S(\G)$ is closed in $\Map(\G)$.
\end{proposition} 

\begin{proof}

By ~\Cref{prop:conjugate}, we can argue for a particularly chosen PHE $\phi: S \to \Gamma$. According to \Cref{cor:defretraction_restrict}, we can assume $\phi$ is a deformation retraction such that there is an exhaustion $\{\G_i\}$ of $\G$ by finite graphs and an exhaustion $\{S_i\}$ of $S$ such that the restriction $\phi|_{S_i}:S_i \to \G_i$ is a proper deformation retraction. We can also assume the PHE inverse $\psi: \G \to S$ is inclusion.

We  claim that the image of $U_{S_i}$ via $\Theta$ is contained in $U_{\Gamma_i} \cap \Map_S(\G)$. Note (1) holds by our choice for the PHE. The additional properties hold for a homeomorphism on a surface and thus for an element in the image. Since $\bigcap_i U_{S_i}$ and $\bigcap_i U_{\Gamma_i}$ are each the identity, we have that a sequence converging to the identity in $\Map(S)$ is sent by $\Theta$ to a sequence converging to the identity in $\Map(\Gamma)$, and since these are metrizable topological groups, continuity of $\Theta$ follows. The continuity of $\Theta^\star$ follows immediately.

Now we analyze the preimage of $U_{\Gamma_i} \cap \Map_S(\G)$ via $\Theta$. Any element in the preimage admits a representative homeomorphism $f$ that setwise preserves $S_i$. Otherwise, there exists a curve $\alpha \subset S_i$ such that $f(\alpha)$ is not contained in $S_i$ up to homotopy. Consequently, the curve $\phi(\alpha) \subset \G_i$ is mapped by $\Theta(f)$ to a curve that is not contained in $\G_i$ up to homotopy; indeed, $\Theta(f)\left(\phi(\alpha)\right)$ is homotopic to $\phi(f(\alpha))$, which is not contained in $\G_i$ up to homotopy, as $f(\alpha)$ is not contained in $S_i$ up to homotopy. Then, by the argument in the proof of \Cref{thm:kernel_regular}, we have $f|_{S_i}$ is homotopic to a product of Dehn twists about the (compact) boundary components of $S_i$. 

It follows the preimage of $U_{\Gamma_i} \cap \Map_S(\G)$ via $\Theta$ is contained in $\mathcal{U}_i$, the union  of countably many translates of $U_{S_i}$ by products of Dehn twists about (compact) components of $\partial S_i$. Note some of these Dehn twists are about boundary components of $S_i$ are different than the original compact boundary components of $S$. However, the intersection $\bigcap_i \mathcal{U}_i$ is the group topologically generated by Dehn twists about only the original boundary components. It follows, a sequence converging to the identity in $\Map_S(\G)$ is mapped by $(\Theta^\star)^{-1}$ to a sequence converging to the identity.

We now have that $\Theta^\star$ is a topological isomorphism, so the image $\Map_S(\G)$ is a Polish group. A subgroup of a Polish group is Polish iff it is closed, so $\Map_S(\G)$ is closed in $\Map(\G)$.\end{proof}

Together, \Cref{thm:kernel_regular} and \Cref{prop: closedimage} make up \Cref{thm:MainTHM_DNB} from the introduction.

\subsection{The extended mapping class group}

For surfaces with interesting boundary, it is natural to expand our discussion to the \emph{extended} mapping class group, $\ExMap(S)$. We define $\ExMap(S)$ as simply the homeomorphism group of $S$ modulo isotopy. So, the homeomorphisms are allowed to act nontrivially on boundary components and the equivalence relation allows isotopies which are nontrivial on the boundary. One effect of this change is that Dehn twists about compact boundary components become trivial mapping classes because they are isotopic to the identity. However, when a surface with nontrivial $\pi_1$ contains punctured boundary components (a boundary chain of degenerate ends) the Dehn twist about the boundary chain is nontrivial. When there are finitely many punctures, there are also elements with the same support as the Dehn twist whose finite powers are Dehn twists, called \emph{fractional Dehn twists}. See \Cref{fig:fractional_DT_simpler} for an example of a $\frac13$-Dehn twist. When there are infinitely many punctures, the types of twists and shifts supported in a neighborhood of the boundary chain depend on the topology of the points removed from the boundary, but there is always a full Dehn twist.

For surfaces of low-complexity, their extended mapping class groups recover a familiar group as in the following example.

\begin{example}[Dihedral group as extended mapping class group] \label{ex: dihedral}
Let $S$ be a disk with a finite number $n$ of boundary points removed. The mapping class group is trivial, but the extended mapping class group is $D_n$, the dihedral group of order $2n$, since $S$ can be viewed as a regular polygon with $n$ (open) vertices.

Removing an \emph{infinite} set from the boundary can have different effects on the extended mapping class group depending on the topology of the infinite set. For example, removing a set homeomorphic to $\omega+1 \cong \{0\} \cup \{\frac{1}{n} \mid n \in \N \}$ yields a surface with trivial extended mapping class group (note there is a unique boundary component going off to the limit point, and once this is fixed all other boundary components are as well). On the other hand, if we remove two disjoint sets each homeomorphic to $\omega +1$ then there are two possible surfaces depending on the direction in which the points accumulate to the two limit points. See \Cref{fig:nonhomeo2}.
  One of these surfaces has $D_\infty$, the infinite dihedral group, as its extended mapping class group, and the other has $\Z_2$ as its extended mapping class group. However, each of these two surfaces is PHE to the same graph (and thus to each other) demonstrating that different surfaces can produce different maps to the same $\Map(\G)$. 
\end{example}

There is a natural map
\[
    \widetilde{\Theta}: \ExMap(S) \longrightarrow \Map(\G), \qquad [f] \mapsto [\phi f \psi].
\]
Following the same proof as \Cref{prop:HomMappingClassGroups} we have the following proposition

\begin{proposition} \label{prop:HomExtMappingClassGroups}
    The map $\widetilde{\Theta}: \ExMap(S) \to \Map(\G)$ is a well-defined homomorphism. 
\end{proposition}

In this case, the image of $\widetilde{\Theta}$ may contain elements that do not fix the ends of $\G$. As an example, consider a disk with three points removed from the boundary, which is PHE to a tripod graph. An order $3$ rotation around the boundary of the disk is mapped to a cyclic permutation of the three ends of the tripod graph via $\widetilde{\Theta}$. 

Our goal is to show $\widetilde{\Theta}$ is injective. Note a proper homotopy equivalence $S \to \G$ maps a compact boundary component of $S$ to a loop in $\G$, and an oriented noncompact boundary component of $S$ to an oriented line in $\G$. We will need this correspondence to be `injective' as well, to apply \Cref{lem:AB_Alexander}. Namely, we need a compact boundary component of $S$ to be mapped to a nontrivial loop of $\G$, distinct compact boundary components are mapped to distinct curves of $\G$ up to homotopy, and distinct oriented noncompact boundary components of $S$ are mapped to distinct oriented lines in $\G$ up to proper homotopy.

\begin{examples} \label{exa:sporadic} We first discuss the surfaces for which $\widetilde{\Theta}$ fails to be injective. 
\begin{enumerate}

    \item Consider the disk. The extended mapping class of the disk is $\Z_2$. The disk is PHE to a point, which has a trivial mapping class group. Indeed, it has one compact boundary component which is mapped to a point by the PHE, which is a trivial loop.
    
    \item Consider the half-strip, $\R_{\ge 0} \times [0,1]$. The extended mapping class group is $\Z_2$, but the half-strip is PHE to a ray that has trivial mapping class group. Indeed, the half-strip has two distinct \emph{oriented} noncompact boundary components, and both are mapped to the same oriented line in the ray by the PHE.
    
    \item Consider the strip, $\R \times [0, 1]$. The extended mapping class group is $D_2 \cong \Z_2 \times \Z_2$, but the strip is PHE to a line whose mapping class group is $\Z_2$. The kernel in this case is generated by the reflection that swaps the two boundary components of the strip. Indeed, the strip has two distinct noncompact boundary components (up to orientation) and both are mapped to the same line by the PHE.
    
    \item The extended mapping class group of the annulus is $\Z_2 \times \Z_2$ (one coordinate records the permutation of the boundary components, and the other records the orientation, while the Dehn twist is trivial). The annulus is PHE to a circle whose mapping class group is $\Z_2$, and the kernel is generated by the reflection that swaps boundary components. Indeed, in this case there are two compact boundary components, both of which are mapped to the same loop by the PHE.

    \end{enumerate}

\end{examples}
We call the surfaces homeomorphic to the surfaces in \Cref{exa:sporadic} --- the disk, the half-strip, the strip, or the annulus ---  \textbf{sporadic}. Otherwise, we call the surface  \textbf{nonsporadic}.

\begin{theorem} \label{thm:kernel}
     Suppose $S$ is a nonsporadic surface PHE to a graph $\G$. The map $\Tilde{\Theta}: \ExMap(S) \rightarrow \MCG(\G)$ is injective.
\end{theorem} 
\begin{proof}

We first claim that for a nonsporadic surface $S$ PHE to a graph, the boundary components of $S$ are mapped injectively. Note $S$ is sufficiently complicated so we assume it has a complete hyperbolic metric with geodesic boundary. 

By \cite[Theorem 1.7]{Epstein1966}, if any compact boundary component of $S$ is homotopic to a point, then it must bound a disk, a contradiction to the assumption that $S$ is nonsporadic.

Suppose two distinct compact boundary components of $S$ are homotopic. Following \cite[Proposition 1.10]{primer}, we lift the corresponding boundary geodesics to geodesics with the same endpoints in the universal cover, but since the universal cover is a region of hyperbolic space, we must have that the boundary components lift to the same geodesic, a contradiction.

Since we are considering surfaces with boundary, note the universal cover is a region of the hyperbolic plane bounded by geodesics. If two of these geodesics are proper homotopic within the universal cover (or a geodesic is proper homotopic to its reverse), then there cannot be any other boundary geodesics separating their endpoints. It follows the only possible cases are the half-strip bounded by a single geodesic and a strip bounded by two, a contradiction. The claim now follows.

Suppose $[f]$ is in the kernel of $\Tilde{\Theta}$. By \Cref{lem:AB_Alexander}, $\Tilde{\Theta}([f])$ fixes oriented curves in $\G$ up to homotopy and fixes oriented lines going out to the ends up to proper homotopy. The noncompact boundary components of $S$ are mapped to oriented lines going out to the ends in $\G$ via the PHE from $S$ to $\G$. Since $S$ is nonsporadic, in particular it is not homeomorphic to the half-strip or the strip, we must have that $f$ setwise fixes each oriented noncompact boundary component of $S$. On the other hand, compact boundary components of $S$ are mapped to loops in $\G$. Similarly, as $S$ is nonsporadic, in particular not homeomorphic to the disk or the annulus, the compact boundary components must be also fixed setwise since oriented curves are fixed up to homotopy by $f$. Now by an argument similar to the proof of \Cref{lem:twists_regular}, we have that $f$ can be written as a possibly infinite product of Dehn twists about degenerate chains, but these twists are not in the kernel so that $f$ must actually be the identity.
\end{proof}

Now using the same proof as \Cref{prop: closedimage}, we have the following. 

\begin{proposition}\label{prop:closedimage_EX}
    The map $ \widetilde{\Theta}$ is continuous and has a closed image.
\end{proposition}

Together \Cref{prop:HomExtMappingClassGroups}, \Cref{thm:kernel}, and \Cref{prop:closedimage_EX}
prove \Cref{thm:MainTHM_DNB_EX}.

\section{The image of the induced map} \label{sec:image}

Now we discuss progress towards answering \Cref{q:DNBimage} and
\Cref{conj:mainimage}. Consider a set $\mathcal{B}$ of properly immersed oriented lines and immersed oriented curves in $\G$. Define $\Map(\G,\mathcal{B})$ as the subgroup of mapping classes on $\G$ that fixes each class in $\mathcal{B}$ (and its orientation) up to proper homotopy. Let $\Map^\star(\G,\mathcal{B})$ denote the subgroup of mapping classes that preserve the lines and curves in $\mathcal{B}$ setwise (but possibly reverse the orientations). 
For the rest of the section, fix the collection $\mathcal{B}$ of lines and curves in $\G$ that are the images of boundary components of $S$ via the proper deformation retraction $S \to \G$.

\begin{proposition}\label{prop:imgfix}
  Let $S$ be a surface PHE to a graph $\G$.
The image of $\Theta:\Map(S)\to\Map(\G)$ is contained in $\Map(\G,\mathcal{B})$, and the image of $\Tilde{\Theta}: \ExMap(S) \to\Map(\G)$ is contained in $\Map^\star(\G,\mathcal{B})$.
\end{proposition}

\begin{proof}
In the first case, the proposition follows from the fact that mapping class group elements have representatives that fix the boundary pointwise. The second case follows from the fact that extended mapping class group elements have representatives that fix the boundary setwise.
\end{proof}

\begin{example}\label{exa:DNBimage1ray}
    Consider when $S$ is a compact surface with one boundary component, and
    $\hat{S}$ is obtained from $S$ with a single point removed from the
    boundary, so that $\Map(S) \cong \Map(\hat{S})$. Call the resulting boundary
    end $p$. Let $\G$ be a graph PHE to
    $S$, and $\hat{\G}$ be $\G$ modified by attaching a ray, so that $\hat{\G}$
    is PHE to $\hat{S}$. We have that $\Map(\G) \cong \Out(F_n)$ while
    $\Map(\hat{\G}) \cong \Aut(F_n)$ by \Cref{lem:PMapfinitegraph}. Choose the PHE $\hat{S} \to \hat{\G}$ so that it restricts to the PHE $S \to \G$. Such a choice results in the induced map $\hat{\Theta}: \Map(\hat{S}) \rightarrow \Map(\hat{\G})$ composed with the quotient map $q:\Aut(F_n) \rightarrow \Out(F_n)$ agrees with $\Theta$. Here $q$ is exactly the map $f \mapsto \pi f i$ where $i$ is the inclusion of $\G$ to $\hat{\G}$ and $\pi$ is a retraction $\hat{\G}$ to $\G$ collapsing the ray.
    
    We have by \Cref{thm:MainTHM_DNB,thm:MainTHM_DNB_EX} that $\hat{\Theta}$ is injective, while $\Theta$ has kernel generated by the Dehn twist about the boundary component which acts by conjugating by the element of $F_n$ coming from the curve around the boundary component. By the usual DNB theorem, \Cref{thm:classic_DNB}, the image of $\Theta$ is the group of outer automorphisms \emph{induced from the orientation preserving homeomorphisms}, fixing the conjugacy class of the curve \emph{and its orientation} coming from the boundary. It follows that the image of $\hat{\Theta}$ in $\Aut(F_n)$ is the group of \emph{orientation preserving} (that is, its determinant after abelianization is positive) automorphisms fixing the elements of $F_n \cong \pi_1(\hat{S}, p)$ coming from the noncompact boundary component. Here, $\pi_1(\hat{S}, p)$ is the group of homotopy classes of lines coming from and going to the end $p$, introduced in \cite[Section 3]{AB2025}. When there are additional compact boundary components, the same idea shows that the image of $\Theta$ is the group of automorphisms preserving the line \emph{with its orientation} from the punctured boundary component viewed as a free group element and preserving the curves \emph{with their orientations} viewed as conjugacy classes.
\end{example}

\begin{proposition}\label{prop:image_finiterays}
For a surface $S$ that is PHE to a finite graph $\G$ with finitely many rays attached, the image of $\Theta:\Map(S) \to \Map(\G)$ is equal to $\Map(\G,\mathcal{B})$. 
\end{proposition}

\begin{proof} 
Note $\Map(\G,\mathcal{B})$ is in $\PMap(\G)$, so we work within the pure
mapping class group throughout the proof. We argue by induction on the number of
rays. The base case was covered in \Cref{exa:DNBimage1ray}. Hence, say $\G$ has
$e$ ends, and we may assume $e \ge 2$. Let $n=\rk(\G)$. We fill in a puncture on
the boundary of $S$ to obtain $\check{S}$, and we modify the graph $\Gamma$ to
obtain $\check{\G}$ by \emph{removing} a ray. The mapping class group of the surface is
unchanged, but the mapping class group of the graph loses a factor of $F_n$ in
its semidirect product structure. Namely, $\Map(S)=\Map(\check{S})$ and $\PMap(\G)
\cong F_n^{e-1} \rtimes \Aut(F_n)$, so $\PMap(\check{\G}) \cong  F_n^{e-2} \rtimes
\Aut(F_n)$ by \Cref{lem:PMapfinitegraph}.

From \cite[Lemma 5.2]{domat2025generating}, there is a \emph{forgetful} homomorphism $\phi: \PMap(\G) \rightarrow \PMap(\check{\G})$ with kernel $F_n$ represented by word maps supported on the ray, which are homotopy equivalences given by wrapping the missing ray around a given word in $F_n$, and a section $i: \PMap(\check{\G}) \rightarrow \PMap(\G)$ given by inclusion of graphs.
Set up the PHE $S \to \G$ so that it restricts to a PHE $\check{S} \to \check{\G}$. Then the induced map $\check{\Theta}: \Map(S)=\Map(\check{S}) \rightarrow \PMap(\check{\G})$ agrees with $\phi \circ \Theta$, and $\phi$ is injective on the image of $\Theta$, as there are no word maps supported on the (forgotten) ray in the image of $\Theta$. Here $\phi$ is exactly the map $f \mapsto \pi f i$ where $i$ is the inclusion of $\check{\G}$ to $\G$ and $\pi$ is a retraction $\G$ to $\check{\G}$ collapsing the ray.

By induction hypothesis, the image of $\check{\Theta}$ is exactly the elements fixing the lines from punctured boundary components and curves from the compact boundary components of $\check{S}$. Call this collection of lines and curves $\check{\mathcal{B}}$. Now suppose $g \in \Map(\G,\mathcal{B})$. Using the section $i$, we conflate $\Map(\check{\G})$ with a subgroup of $\Map(\G)$. We claim $\phi(g) \in \Map(\check{\G},\check{\mathcal{B}})$, so that it is in the image of $\check{\Theta}=\phi \circ \Theta$ by the hypothesis. Suppose we have shown the claim; write $\phi(g) = \phi(\Theta(f))$, for some $f \in \Map(S)$. This implies $g = k\Theta(f)$ for some $k \in \ker \phi$. This implies $k = g\Theta(f)^{-1} \in \Map(\G,\mathcal{B})$ by \Cref{prop:imgfix}. However, as the nontrivial elements in $\ker \phi$ act nontrivially on a ray in $\mathcal{B}$, this implies $\ker \phi \cap \Map(\G,\mathcal{B})=1$. In particular, $k=1$. Hence, it follows that $g$ is in the image of $\Theta$, showing $\Map(\G,\mathcal{B})$ is contained in the image of $\Theta$. The other containment is given by \Cref{prop:imgfix}, so the claim will conclude the proof.

To show the claim, first note that $\check{\mathcal{B}}$ is related to $\mathcal{B}$ in the following manner. By filling in the puncture and removing the ray, either one line in $\mathcal{B}$ is surgered into a curve in $\check{\mathcal{B}}$ (here the line corresponds to a noncompact boundary component going to and from the same end) or two lines in $\mathcal{B}$ are surgered into a line in $\check{\mathcal{B}}$. Every curve and line in $\check{\mathcal{B}}$ is either obtained from $\mathcal{B}$ in this way or is also in $\mathcal{B}$. Since $g$ fixes the elements of $\mathcal{B}$, and in particular $g$ fixes the given ray, $\phi(g)$ is just restriction of $g$ to $\check{\G}$. Hence, by the observation of elements of $\check{\mathcal{B}}$ obtained from those of $\mathcal{B}$, it follows that $\phi(g)$ must fix the elements of $\check{\mathcal{B}}$. \end{proof}

A PHE $\phi$ of a graph $\G$ is said to be \textbf{totally supported} on a subgraph $\Delta$ of $\Gamma$ if there exists a representative of $\phi$ supported on $\Delta$ and $\phi(\Delta) \subseteq \Delta$, equivalently if there exists a proper homotopy inverse $\psi$ of $\phi$ whose representative is supported on $\Delta$.
Denote by $\Map_\G(\Delta)$ the subgroup of $\Map(\G)$ of elements totally supported on a subgraph $\Delta$. Note $\Map_\G(\Delta)$ can be naturally identified with $\Map(\Delta,\partial \Delta)$, the group of PHEs of $\Delta$ relative to the frontier $\partial\Delta$ of $\Delta$ in $\G$. Recall $\partial \Delta = \Delta \cap \overline{\G \setminus \Delta}$. Similarly, for a subsurface $K$ of $S$, $\Map(K)$ is identified with the subgroup of $\Map(S)$ of elements supported on a compact subsurface $K$. 

Given a compact subgraph $\Delta \subset \G$, denote by $\Delta^\dag$ the graph obtained from $\Delta$, attaching a ray to each point in $\partial \Delta$. Then we have a proper map $c: \G \to \Delta^\dag$, defined as identity on $\Delta$ but collapsing each complementary component of $\Delta$ in $\G$ to the corresponding ray in $\Delta^\dag \setminus \Delta$. On the other hand, given $g \in \Map(\G)$ totally supported on $\Delta$, there is a natural element $g^\dag \in \Map(\Delta^\dag)$ where $g^\dag=g$ on $\Delta$, and identity otherwise.
Regarding preserving the proper lines in $\G$, we have the following consistency.

\begin{lemma}\label{lem:line-segment}
    Let $\Delta \subset \G$ be a compact subgraph and $g \in \Map(\G)$ be totally supported on $\Delta$. If $\ell$ is a proper line in $\G$ so that $\ell \setminus \Delta$ is an interval, then
    \[
        g(\ell) \simeq \ell \text{ in $\G$} \quad \text{$\Longleftrightarrow$} \quad g^{\dag}(c(\ell)) \simeq c(\ell) \text{ in $\Delta^\dag$},
    \]
    where $\simeq$ denotes proper homotopy. Equivalently, $g$ fixes the proper homotopy class of $\ell$ in $\G$ if and only if $g$ fixes the proper homotopy class of each segment in $\ell \cap \Delta$, rel $\partial \Delta$.
\end{lemma}

\begin{proof}
    $(\Leftarrow)$ We may take a proper homotopy $g^\dag(c(\ell)) \simeq c(\ell)$ to be stationary on the attached rays outside $\Delta$ in $\Delta^\dag$. Then this induces a proper homotopy $g(\ell)\simeq \ell$ in $\Delta$ in $\G$, and we can extend by identity on the complementary component of $\Delta$ in $\G$.

    $(\Rightarrow)$ Since $g$ is totally supported on $\Delta$ and $g^\dag$ is identity on $\Delta^\dag \setminus \Delta$, we have $c \circ g = g^\dag \circ c$. Since $c$ is proper, by the assumption we have $c(g(\ell)) \simeq c(\ell)$. However, $c(g(\ell)) = g^\dag(c(\ell))$, so we establish the equivalence.

    The latter statement follows from that the proper homotopy $g^\dag(c(\ell)) \simeq c(\ell)$ can be taken stationary on $\Delta^\dag \setminus \Delta$, and that the proper homotopy on the segment in $\ell \cap \Delta$ can be extended by identity.
\end{proof}

Now we extend \Cref{prop:image_finiterays} by showing the following. Let $\Map_c(S)$ and $\Map_c(\G)$ be the groups of compactly supported elements, as subgroups of $\PMap(S)$ and $\PMap(\G)$ respectively.

\begin{theorem}\label{thm:image_compactsupp}
    For a surface $S$ that is PHE to a graph $\G$,
    the image of $\Theta:\Map(S) \to \Map(\G)$ \emph{restricted to} $\overline{\Map_c(S)}$ is equal to $\overline{\Map(\G,\mathcal{B}) \cap \Map_c(\G)}$. 
\end{theorem}

\begin{proof}

We first show that an element of the image is contained in $\overline{\Map(\G,\mathcal{B}) \cap \Map_c(\G)}$. Consider some $g \in \overline{\Map_c(S)}$ and a sequence of elements $g_i \in \Map_c(S)$ converging to $g$. By continuity, $\Theta(g_i)$ converges to $\Theta(g)$. Now note each $\Theta(g_i)$ is in both $\Map_c(\G)$ and $\Map(\G,\mathcal{B})$ so $\Theta(g)$ must be in $\overline{\Map(\G,\mathcal{B}) \cap \Map_c(\G)}$.

Let $g \in \overline{\Map(\G,\mathcal{B}) \cap \Map_c(\G)}$. We want to show that $g$ is in the image of $\Theta$. Say $\{g_i\}_{i=1}^\infty$ converges to $g$ for some $g_i \in \Map(\G,\mathcal{B}) \cap \Map_c(\G)$.
According to \Cref{cor:defretraction_restrict}, take an exhaustion $\{\G_i\}$ of $\G$ by finite graphs and an exhaustion $\{S_i\}$ of $S$ by compact subsurfaces such that $\Theta$ restricts to homomorphisms $\Theta_i: \Map(S_i) \to \Map_\G(\G_i)$. Relabeling if needed, we may assume each $g_i$ is totally supported in $\G_i$ for each $i \ge 1$. We further assume, after possibly modifying both exhaustions, that boundary 
components of the $S_i$ intersecting $\partial S$ either 1) agree with compact boundary components of $S$ or 2) alternate between arcs in noncompact boundary components of $S$ and arcs in the interior of $S$. We can further assume each boundary component of the $S_i$ intersects any noncompact component of $\partial S$ in at most one arc.

We now view $\Map_\G(\G_i)$ as $\Map(\G_i,\partial \G_i)$, and lines in $\G_i$ will refer to homotopy classes rel $\partial \G_i$ of intervals with endpoints in $\partial \G_i$. Note we may regard $\Map(\G_i,\partial \G_i)$ as $\PMap(\G_i^\dag)$. By \Cref{prop:image_finiterays}, the image of $\Theta_i$ consists of the PHEs of $\G_i$ preserving the homotopy class of each line and curve coming from the $\partial S_i$.
We claim that each $g_i$ viewed as an element of $\Map(\G_i,\partial\G_i)$ is in the image of $\Theta_i$, after possibly enlarging $S_i$ and $\G_i$. That is, $g_i$ fixes the homotopy classes of curves and lines coming from $\partial S_i$ rel $\partial \G_i$. 

First, $g_i$ fixes the curves in $\mathcal{B}$ as  $g_i$ is in $\Map(\G,\mathcal{B})$. Additional curves come from boundary components of $S_i$ that do not intersect $\partial S$, but we can assume the homotopy classes determined by them are fixed by enlarging $S_i$ and $\G_i$.

Now suppose for the sake of contradiction that there is a line (interval) $\ell$ from $\partial S_i$ rel $\partial \G_i$ that is not fixed by $g_i$. Then by definition $\ell$ can be extended to a loop or a line that is homotopic to the image of a boundary component of $S$ via the deformation retraction. For the former, we already showed $g_i$ fixes every loop up to homotopy, so up to homotopy we can make $g_i$ fix $\ell$ rel $\partial \G_i$, contradicting to the choice of $\ell$. It remains to considering the latter case; extend $\ell$ to a line in $\mathcal{B}$, but as $g_i$ is totally supported on $\G_i$, the extended line is not fixed by $g_i$, contradicting $g_i \in \Map(\G,\mathcal{B})$ by \Cref{lem:line-segment}.

Finally, we have that each $g_i$ has some $f_i \in \Map(S_i)$ such that $\Theta(f_i)=\Theta_i(f_i) = g_i$. Since the $g_i$ converge to $g$ and the image of $\Theta$ is closed by \Cref{prop: closedimage}, we must have $g=\lim_{i \to \infty} g_i = \lim_{i \to \infty} \Theta(f_i)$ is in the image of $\Theta$.
\end{proof}

Recall that for surfaces PHE to a graph, the pure mapping class group agrees with the mapping class group. By Patel--Vlamis \cite{PV2018} and the extension to the surfaces with noncompact boundary by Dickmann \cite[Theorem 6.13]{dickmann2023mapping}, the pure mapping class group of surfaces fits into the following split short exact sequence.

\[
1 \longrightarrow \overline{\Map_c(S)} \longrightarrow \PMap(S) \longrightarrow \Z^{n-1}\longrightarrow 1
\]
where $n$ is the number of ends of the surface accumulated by genus, using the convention that $\infty-1=\infty$. Furthermore, any element of $\PMap(S)$ can be written as a product of an element in $\overline{\Map_c(S)}$ with a product of \textbf{handle shift} maps. A handle shift map has a representative supported in a strip with genus, a strip $\R \times [-\frac{1}{2},\frac{1}{2}]$ modified by attaching genus at each integer point (see top of \Cref{fig:doubleflux}). This representative is formed by shifting the strip by 1 in the $\R$ coordinate and then homotoping this shift map so that the boundary is fixed pointwise. 

The pure mapping class group for infinite graphs was studied by Domat--Hoganson--Kwak \cite{domat2025generating} who showed the existence of the following split exact sequence 

\[
1 \longrightarrow \overline{\Map_c(\G)} \longrightarrow \PMap(\G) \longrightarrow \Z^{n-1}\longrightarrow 1
\]
where $n$ is the number of ends of the graph accumulated by loops.

For surfaces and for graphs, the homomorphisms from the pure mapping class group to $\Z^{n-1}$ are referred to as the \textbf{flux} homomorphisms. The even flux subgroup will refer to the subgroup with even entries in each coordinate.

\begin{figure}[ht!]
    \centering
    \includegraphics[width=0.65\textwidth]{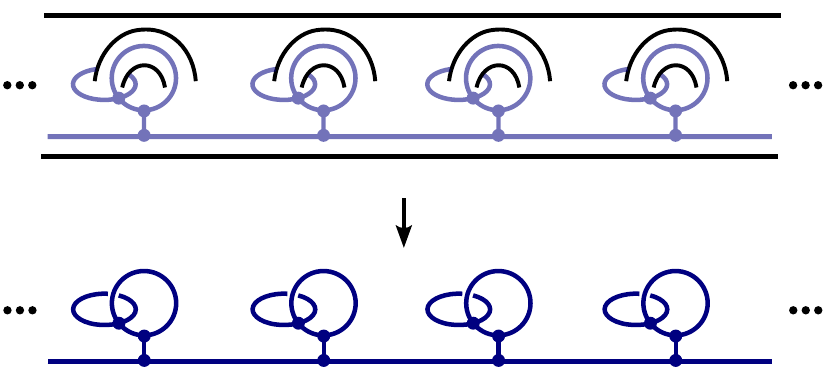}
    \caption{A strip with genus above deformation retracts to the graph below. A handle shift is mapped to a loop shift moving each pair of loops to an adjacent pair.}
    \label{fig:doubleflux}
\end{figure}

\begin{proposition}
\label{prop:doubleflux}
    The map $\Theta$ doubles the flux of a given map, and thus its image is contained in the even flux subgroup of $\PMap(\G)$. 
\end{proposition}

\begin{proof}
    Suppose that the surface has at least two ends accumulated by genus since
    that is the interesting case. Let $f \in \Map(S)$ be a given element, and
    write it as a product of an element in $\overline{\Map_c(S)}$ with disjoint
    handle shifts. Since $\Theta$ maps $\overline{\Map_c(S)}$ to the zero flux
    subgroup $\overline{\Map_c(\G)}$ by \Cref{thm:image_compactsupp}, it suffices to show any handle shift is sent to a map with even flux. See \Cref{fig:doubleflux} for an example. Here, we took $S \to \G$ the proper deformation retraction as in \Cref{cor:defretraction_restrict} so that the \emph{reference domain} for the flux for $S$ descends to that for $\G$ under the proper deformation retraction. In effect, if a handle shift on $S$ has flux $n$, then $\Theta(f)$ has flux $2n$ as each genus yields rank 2 subgraph after the proper deformation retraction. This settles the even flux claim for \Cref{fig:doubleflux}.
    
    We now argue we can reduce to the case in the figure. In general, any handle shift has a representative supported in a properly embedded strip with genus. The map $\Theta$ is defined by a chosen PHE, but we are free to choose a different PHE. By \Cref{prop:conjugate}, the new homomorphism will be conjugate to the original, and it suffices to work with a conjugate since the even flux subgroup is normal. We then choose the new deformation retraction of $S$ to $\G$ so that the restriction to the properly embedded strip with genus is a deformation retraction to a subgraph of $\G$ in a similar fashion to \Cref{fig:doubleflux}. The only difference will be that the subgraph has additional edges leaving the strip with genus to connect to the remainder of the graph.\end{proof}

\bibliography{bib}
\bibliographystyle{alpha}

\end{document}